\documentclass[12pt]{amsart}
\usepackage[margin=1.5in]{geometry}
\usepackage{tikz,tikz-cd}
\usepackage{url}
\usepackage{graphicx}
\usepackage{graphicx}
\usepackage{color}
\usepackage{times}
\usepackage{cite}
\usepackage{enumerate,latexsym}
\usepackage{latexsym}
\usepackage{amsmath,amssymb}% CUIDADO: Si uso estos packages, $\widetilde{}$ no funciona a veces
\usepackage{graphicx}
\usepackage{amsthm}
\usepackage{verbatim}
\newtheorem{thm}{Theorem}[section]
\newtheorem*{theorem*}{Theorem}
\newtheorem*{acknowledgement*}{Acknowledgement}

\newtheorem{lem}[thm]{Lemma}
\newtheorem{prop}[thm]{Proposition}
\theoremstyle{definition}
\newtheorem{defn}[thm]{Definition}
\theoremstyle{remark}
\newtheorem{rem}[thm]{Remark}

\newtheorem{conj}[thm]{Conjecture}
\numberwithin{equation}{section}

\newcommand{\set}[1]{\left\{#1\right\}}
\newcommand{\Real}{\mathbb R}

\newcommand{\dist}[0]{\mathrm{dist}}

\title[The CR-Volume of horizontal submanifolds]{The CR-Volume of Horizontal Submanifolds of Spheres}

\author{Jacob Bernstein}
\address{Department of Mathematics, Johns Hopkins University, 3400 N. Charles Street, Baltimore, MD 21218}
\email{bernstein@math.jhu.edu}

\author{Arunima Bhattacharya}
\address{Department of Mathematics, the University of North Carolina at Chapel Hill, Phillips Hall, Chapel Hill, NC 27514 }
\email{arunimab@unc.edu}

\begin{document}

\begin{abstract}
We study an analog in CR-geometry of the conformal volume of Li-Yau. In particular, to submanifolds of odd-dimensional spheres that are Legendrian or, more generally, horizontal with respect to the sphere's standard CR-structure we associate a quantity that is invariant under the CR-automorphisms of the sphere.  
 We apply this concept to a corresponding notion of Willmore energy.
\end{abstract}
\maketitle

\section{Introduction}
In   \cite{Li1982}, Li-Yau introduced a conformal invariant called conformal volume and used it to study the Willmore energy and spectral properties of surfaces.   Inspired by this, we consider an analog in CR-geometry which we term \emph{CR-volume}.
In particular, by endowing the odd-dimensional sphere, $\mathbb{S}^{2n+1}$,  with the usual Sasaki and CR-structures it inherits as the boundary of the ball $\mathbb{B}^{2n+2}\simeq\mathbb{B}^{n+1}_{\mathbb{C}}$,  one may associate to every horizontal submanifold $\Sigma\subset \mathbb{S}^{2n+1}$ a corresponding quantity that is invariant under the CR-automorphisms of $\mathbb{S}^{2n+1}$.
  As observed by the authors in \cite{BernBhattCplxHyp}, this quantity arises naturally in the study of certain minimal submanifolds of complex hyperbolic space.  In this article we systematically establish some basic properties of CR-volume and observe that it provides a unifying perspective for a number of different results.
 
Let $Aut_{CR}(\mathbb{S}^{2n+1})$ be the group of CR-automorphisms of the usual CR-structure on $\mathbb{S}^{2n+1}$; these are also called \emph{pseudoconformal} transformations by some authors.  Elements of $Aut_{CR}(\mathbb{S}^{2n+1})$ preserve the \emph{horizontal} distribution, $\mathcal{H}$, of $\mathbb{S}^{2n+1}$ associated to the CR-structure -- see \eqref{Hp}.  For $1\leq m \leq n$, let $\Sigma\subset \mathbb{S}^{2n+1}$ be an $m$-dimensional \emph{horizontal} (also called \emph{isotropic}) submanifold, i.e., satisfying $T_p\Sigma\subset \mathcal{H}_p$ for all $p\in \Sigma$.  When $m=n$, such a submanifold is called \emph{Legendrian}.  For a horizontal submanifold, $\Sigma$, we define its CR-volume to be:
\begin{equation}\label{CRVolEqn}
	\lambda_{CR}[\Sigma]=\sup_{\Psi\in Aut_{CR}(\mathbb{S}^{2n+1})} |\Psi(\Sigma)|_{\mathbb{S}}
\end{equation}
where $|\cdot |_{\mathbb{S}}$ denotes the volume of the submanifold with respect to the standard round metric $g_{\mathbb{S}}$ on $\mathbb{S}^{2n+1}$.
More generally, when $i: M\to \mathbb{S}^{2n+1}$ is a horizontal immersion of an $m$-dimensional manifold, we define
$$
V_{CR}(2n+1, i)= \sup_{\Psi\in Aut_{CR}(\mathbb{S}^{2n+1})} |M|_{(\Psi \circ i)^* g_{\mathbb{S}}}.
$$
This is analogous to the $n$-conformal volume, $V_c(n, j)$, of an immersion $j: N \to \Gamma\subset \mathbb{S}^{n}$ from \cite{Li1982}. When $i$ is an embedding, $V_{CR}(2n+1, i)=\lambda_{CR}[i(M)]$ and we have adapted in \eqref{CRVolEqn} the notation of  \cite{BernsteinHypEntropy} for the  {conformal volume}, $\lambda_c[\Gamma]=V_{c}(n,j)$, when $\Gamma$  is a submanifold and $j$ is an embbeding.
%$$
%\lambda_c[\Gamma]=\sup_{\Psi \in \mathrm{M{o}b}(\mathbb{S}^{n})} |\Psi(\Gamma)|_{\mathbb{S}} %=V_c(n ,j)
%$$
%where $ \mathrm{Mob}(\mathbb{S}^{n})=Aut_c(\mathbb{S}^n)$ is the group of conformal automorphisms of $\mathbb{S}^{n}$.

Our first result is the following:
\begin{thm}\label{intro_thm_hori}
	If $\Sigma\subset \mathbb{S}^{2n+1}$ is an $m$-dimensional closed horizontal submanifold and $\mathbb{S}^m$ is an  $m$-dimensional  totally geodesic horizontal sphere in $\mathbb{S}^{2n+1}$,  then
	$$
	\lambda_{CR}[\Sigma]\geq \lambda_{CR}[\mathbb{S}^m]=|\mathbb{S}^m|_{\mathbb{S}}.
	$$
	Equality holds if and only if $\Sigma$ is a contact Whitney sphere, i.e., $\Sigma=\Psi(\mathbb{S}^m)$ for some $\Psi\in Aut_{CR}(\mathbb{S}^{2n+1})$.  	Moreover, there is a $\Psi\in Aut_{CR}(\mathbb{S}^{2n+1})$ so the $CR$-volume of $\Sigma$ is achieved, i.e., 
	$$
		\lambda_{CR}[\Sigma]= |\Psi(\Sigma)|_{\mathbb{S}}.
	$$
	More generally, if $i:M\to \mathbb{S}^{2n+1}$ is a horizontal immersion, then 
	$$
	V_{CR}(n, i)\geq k |\mathbb{S}^m|_{\mathbb{S}}
	$$
	where $k$ denotes the number of points in the preimage of $p\in i(M)$.
\end{thm}
\begin{rem}
	As observed in \cite{bryantSurfacesConformalGeometry1988}, for immersions that are not embeddings the conformal volume may not be achieved. The same is true for the CR-volume;  an example is given by the union of two contact Whitney spheres that meet in only one point.
\end{rem}

An important application of the conformal volume is the study of  Willmore energy. This suggests identifying a corresponding energy associated to the CR-geometry of the sphere.   
One such quantity was proposed by Wang in \cite{wangWillmoreLegendrianSurfaces2007} for Legendrian surfaces in $\mathbb{S}^5$.
We consider a slightly different formulation and take the \emph{CR-Willmore energy} of a Legendrian surface $\Sigma\subset \mathbb{S}^5$ to be
\begin{equation} \label{CRWillmoreEqn}
\mathcal{W}_{CR}[\Sigma]= \int_{\Sigma} 1+\frac{1}{8} |\mathbf{H}_\Sigma|^2 dA_\Sigma
\end{equation}
where $\mathbf{H}_\Sigma$ is the mean curvature of $\Sigma$ in $\mathbb{S}^{5}$, which we define as the trace of the second fundamental form and not the averaged trace. The usual Willmore energy is
$$
\mathcal{W}[\Sigma]= \int_{\Sigma} 1+\frac{1}{4} |\mathbf{H}_\Sigma|^2 dA_\Sigma.
$$   
The definitions immediately imply that
$$
\mathcal{W}[\Sigma]\geq\mathcal{W}_{CR}[\Sigma]\geq  |\Sigma|_{\mathbb{S}}
$$
and the inequalities are strict unless $\Sigma$ is minimal.

While $\mathcal{W}$ is conformally invariant,  $\mathcal{W}_{CR}$ is  $Aut_{CR}(\mathbb{S}^5)$ invariant.
\begin{thm}\label{intro_thm_lag}
When $\Sigma \subset \mathbb{S}^{5}$ is a closed Legendrian surface and $\Psi\in Aut_{CR}(\mathbb{S}^5)$,
$$
\mathcal{W}_{CR}[\Sigma]=\mathcal{W}_{CR}[\Psi(\Sigma)] \mbox{ and }\mathcal{W}_{CR}[\Sigma]\geq \lambda_{CR}[\Sigma].
$$
When $\Sigma$ is embedded,  the inequality is strict unless $\Psi(\Sigma)$ is minimal for some $\Psi$.
\end{thm}

\begin{rem}  When $\Sigma$ is a minimal Legendrian submanifold, $\lambda_{CR}[\Sigma]=\mathcal{W}_{CR}[\Sigma]=\mathcal{W}[\Sigma]=\lambda_c[\Sigma]$. 
However, in general the conformal area and the CR-volume of a submanifold may differ -- see Proposition \ref{CRConfAreaDifferProp}.
\end{rem}	

%A consequence is that, in general, the conformal and $CR$-volumes differ.
%\begin{cor}
%	Suppose	$\Sigma \subset \mathbb{S}^{5}$ is a closed embedded Legendrian minimal surface.  There exist, $\Psi\in Aut_{CR}(\mathbb{S}^5)$ so
%	$$
%	\lambda_{CR}[\Psi(\Sigma)]<\lambda_{c}[\Psi(\Sigma)]
%	$$
%	where $\lambda_c[\Sigma]$ is the conformal area of $\Sigma$.
%\end{cor}

We conclude by mentioning an appealing Legendrian analog of the Willmore conjecture.  Recall, there is a Legendrian minimal torus in $\mathbb{S}^5$ given by:
$$
\Sigma_H=\set{(\zeta_1, \zeta_2, \zeta_3): |\zeta_1|^2= |\zeta_2|^2= |\zeta_3|^2=\frac{1}{3}, \zeta_1 \zeta_2 \zeta_3=\frac{\sqrt{3}}{9}}\subset \mathbb{S}^5\subset  \mathbb{C}^3.
$$
The induced metric on  $\Sigma_H$ is conformal to the flat torus $\mathbb{C}/\Lambda_H$
where
$$
\Lambda_H= \set{n+\frac{1+i \sqrt{3}}{2} m:  n, m\in\mathbb{Z}}
$$
is the lattice whose fundamental domain has an angle of $60^\circ$. For this reason, $\Sigma_H$ is referred to as the \emph{hexagonal torus} in \cite{wangWillmoreLegendrianSurfaces2007} and \emph{equilateral torus} in \cite{UrbanoStability} -- it is also called the \emph{Clifford torus} in \cite{HaskinsSL}. There is a homothetic parameterization 
\begin{align*}
\phi_H: \mathbb{C}/\Lambda_{H}\to &\Sigma_H\subset \mathbb{S}^5\subset  \mathbb{C}^3\\
 z=x+iy\mapsto &\frac{1}{\sqrt{3}} \left( e^{\frac{4\sqrt{3}}{3} \pi iy}, e^{ 2  \pi i  \left( x-\frac{\sqrt{3}}{3}y\right)}, e^{-2 \pi i \left(x+\frac{\sqrt{3}}{3}y\right)}\right).
\end{align*}
If $\Pi: \mathbb{S}^5\to \mathbb{CP}^2$ is the Hopf fibration, then $\Sigma_C^{\mathbb{CP}}=\Pi(\Sigma_H)$ is the torus given by
$$
\Sigma_C^{\mathbb{CP}}=\set{[z_1, z_2, z_3]: |z_1|^2=|z_2|^2=|z_3|^2=\frac{1}{3}}\subset \mathbb{CP}^2,
$$
 which is referred to as the generalized Clifford torus in $ \mathbb{CP}^2$ by some authors. 
It is Lagrangian and minimal with respect to $g_{\mathbb{CP}}$, the Fubini-Study metric.  The  map $\Pi|_{\Sigma}:\Sigma_H\to \Sigma_C^{\mathbb{CP}}$ is a three-fold cover.  

A natural analog of the Willmore conjecture in the Legendrian setting is:
\begin{conj}\label{CRWillmoreConj} 
	Let $\Sigma\subset \mathbb{S}^5$ be a Legendrian surface of positive genus, then
	$$
	\mathcal{W}_{CR}[\Sigma]\geq \mathcal{W}_{CR}[\Sigma_H]=\frac{4\sqrt{3}}{3}\pi^2
	$$
	with equality holding if and only if $\Sigma$ differs from $\Sigma_H$ by an isometry of $\mathbb{S}^5$.
\end{conj}
For tori, this is equivalent to a conjecture posed by Wang in \cite{wangWillmoreLegendrianSurfaces2007}.  As observed by Wang,  the CR-geometric analog of stereographic projection and \cite{reckziegelCorrespondenceHorizontalSubmanifolds1988} allows one to associate to every Legendrian submanifold of $\mathbb{S}^5$ an exact Lagrangian immersion into $\mathbb{C}^2$ with equivalent Willmore energy.  As the latter are never embedded \cite[2.3.B]{Gromov}, this is a distinct, if related, problem from the one considered by Minicozzi \cite{Minicozzi}.  In addition, as we discuss in Section \ref{MontielUrbanoSec}, conjecture \ref{CRWillmoreConj} relates, via the correspondence of \cite{reckziegelCorrespondenceHorizontalSubmanifolds1988},  to a question of Montiel-Urbano from \cite{montielWillmoreFunctionalCompact2002}.  The weaker form of their question asks whether $\Sigma_C^{\mathbb{CP}}$ is the minimizer among Lagrangian tori in $\mathbb{CP}^2$ of a Willmore-like energy they introduce for surfaces in $\mathbb{CP}^2$ -- see also \cite{kazhymuratLowerBoundEnergy2018, maEnergyFunctionalLagrangian2018, MNSurvey}.  We observe that conjecture \ref{CRWillmoreConj} would give this for tori Hamiltonian isotopic to $\Sigma_C^{\mathbb{CP}}$.

A related conjecture is that $\Sigma_H$ is the least area Legendrian minimal surface of positive genus.
\begin{conj} \label{CRAreaConJ}
		If $\Sigma\subset \mathbb{S}^5$ is a Legendrian minimal surface of positive genus, then
	$$
	|\Sigma|_{\mathbb{S}}\geq |\Sigma_H|_{\mathbb{S}}=\frac{4\sqrt{3}}{3}\pi^2
	$$
	with equality if and only if $\Sigma$ differs from $\Sigma_H$ by an isometry of $\mathbb{S}^5$.
\end{conj}
As cones over Legendrian minimal surfaces are special Lagrangian cones -- see \cite{HaskinsSL} -- an affirmative answer to this conjecture would provide sharp lower bounds on the density of non-trivial special Lagrangian cones in $\mathbb{C}^3$ -- see also \cite{HaskinsComplexity}.  While the only Legendrian minimal two-spheres in $\mathbb{S}^5$ are the totally geodesic ones  \cite[Theorem B]{HaskinsSL} and \cite{YauMinimalSphere}, there are many examples of positive genus in $\mathbb{S}^5$,  e.g.,  \cite{CarberryMcIntosh, McIntosh, HaskinsSL, HaskinsKapouleas}.  Clearly, conjecture \ref{CRWillmoreConj} implies conjecture \ref{CRAreaConJ}, but as in Marques and Neves's proof of the Willmore conjecture \cite{MarquesNeves} there may be a deeper relationship.  
%One might hope that the techniques of \cite{MarquesNeves} could be adapted to the Legendrian setting as many of the ingredients are present -- for instance a characterization of $\Sigma_H$ in terms of its Morse index \cite{HaskinsComplexity, UrbanoStability}.   

Using ideas from \cite{Li1982, montielMinimalImmersionsSurfaces1986} we show a special case of the two conjectures:
\begin{prop}\label{CRWillmoreHexTorProp}
	If $\Sigma \subset \mathbb{S}^5$ is a Legendrian torus conformal to $\mathbb{C}/\Lambda_H$, then
	$$
	 \mathcal{W}_{CR}[\Sigma]\geq \lambda_{CR}[\Sigma]\geq |\Sigma_H|_{\mathbb{S}}.
	 $$
	 Moreover, one has equality if and only if there is a $\Psi\in Aut_{CR}(\mathbb{S}^5)$ so $\Sigma=\Psi(\Sigma_H)$.
\end{prop}
 \begin{rem}
 For any $\Sigma\subset \mathbb{S}^5$ conformally parameterized by the flat torus $\mathbb{C}/\Lambda_H$, $$
 	\mathcal{W}[\Sigma]\geq \lambda_c[\Sigma]\geq \lambda_c[\Sigma_H].$$ This is shown in \cite[Corollary 6]{montielMinimalImmersionsSurfaces1986}.  See also  \cite{bryantConformalVolume2tori2015}.
 \end{rem}
 
 %Finally, we observe that El Soufi-Ilias in \cite[Corollary 2.2]{elsoufiRiemannianManifoldsAdmitting2000} showed that the conformal parameterization of the Clifford torus by the quotient of the square lattice and the conformal parameterization of the quotient of the hexagonal lattice are the only $\lambda_1$-extremizers \note{do we need to clarify the $\lambda_1$ notation?}  of genus one.   

%Some other results that follow from known results:
%\begin{prop}
%	\begin{enumerate}
%	\item  $V_{CR}(\mathbb{S}^2, [g_{\mathbb{S}}], 2n+1)=4\pi$, $n\geq 1$;
%	\item $V_{CR}(\mathbb{RP}^2, [g_{\mathbb{RP}}], 2n+1)=12 \pi, n\geq 4$;
%	\item $V_{CR}(\mathbb{T}^2, [g_{Sq}], 2n+1)= 2\pi^2, n\geq 3$, where here $g_R$ is the metric of the Square torus.	\end{enumerate}
%\end{prop}

\subsection*{Acknowledgements} The authors are grateful to Robert Bryant and Mark Haskins for their helpful comments and suggestions. 
The first author was supported by the NSF Grant DMS-1904674 and  DMS-2203132 and the Institute for Advanced Study with funding provided by the Charles Simonyi Endowment.
The second author acknowledges the support of the Simons Foundation Grant MP-TSM-00002933 and funding from the Bill Guthridge distinguished professorship fund.

	\section{Geometric background}
For the readers' convenience, we summarize some of the relevant background -- see also Sections 2.1 and 2.2 of \cite{BernBhattCplxHyp}.
In particular, we recall the natural contact, Sasaki, and CR-structures that come from viewing $\mathbb{S}^{2n+1}\subset \Real^{2n+2}\simeq \mathbb{C}^{n+1}$  as the boundary of the unit ball $\mathbb{B}^{2n+2}\simeq \mathbb{B}_{\mathbb{C}}^{n+1}$.  Our main sources are \cite{BlairBookNew, Tanno1989}. %\note{check Tanno ref}. 

\subsection{Basic constructs on Euclidean space}

Let us first introduce notation for various structures on  $\Real^{2n+2}\simeq \Real^{n+1}_{\mathbf{x}}\times \Real^{n+1}_{\mathbf{y}}\simeq \mathbb{C}^{n+1}_{\mathbf{z}}$.  Here we make the identification of $(\mathbf{x}, \mathbf{y})\in  \Real^{n+1}\times  \Real^{n+1}$ with $\mathbb{C}^{n+1}$ via
$$
\mathbf{z}=\mathbf{x}+i \mathbf{y}.
$$
In particular, the Euclidean coordinates given by $x_1, \ldots, x_{n+1}, y_1, \ldots,  y_{n+1}$, correspond to the holomorphic coordinates $z_1=x_1+i y_1, \ldots, z_{n+1}=x_{n+1}+iy_{n+1}$.

Denote the usual Euclidean Riemannian metric and symplectic form by
\begin{align*} g_{\Real}=\sum_{j=1}^{n+1} \left( dx_j^2+dy_j^2\right)\mbox{ and } \omega_{\Real}=\sum_{j=1}^{n+1} dx_j\wedge dy_j.
\end{align*}
Let $J_{\Real}$ be the associated almost complex structure on $\Real^{2n+2}$  defined either by
$$
\omega_{\Real}(X,Y)=g_{\Real}(X,J_{\Real}(Y)) \mbox{ or }
$$
$$
J_{\Real}\left(\frac{\partial}{\partial x_j }\right)=-\frac{\partial}{\partial y_j} \mbox{ and } J_{\Real}\left(\frac{\partial}{\partial y_j}\right)=\frac{\partial}{\partial x_j}.
$$
In this convention the complexification of $J_{\Real}$ satisfies $J_{\Real}\left(\frac{\partial}{\partial z_j}\right)=-i \frac{\partial}{\partial z_j}$.
Likewise, 
$$
dx_j\circ J_{\Real}=dy_j \mbox{ and } dy_j \circ J_{\Real} =-dx_j.
$$

Let $r:\Real^{2n+2}\to \Real$ be the radial function defined by
$$
r^2=x_1^2+y_1^2+\cdots+x_{n+1}^2+y_{n+1}^2=|\mathbf{x}|^2+|\mathbf{y}|^2=|\mathbf{z}|^2.
$$
Consider the smooth one-forms on $\Real^{2n+2}\setminus \set{0}$
\begin{align*}
dr &=\frac{1}{r} \left(x_1 dx_1 +y_1 dy_1+\cdots +x_{n+1} dx_{n+1}+y_{n+1} dy_{n+1}\right)\\
\theta&= \frac{1}{r} dr \circ J_{\mathbb{R}}=\frac{1}{r^2} \left(x_1 dy_1 -y_1 dx_1+\cdots +x_{n+1} dy_{n+1}-y_{n+1} dy_{n+1}\right).
\end{align*}
%The normalization ensures the Lie derivative in the radial direction satisfies
%$$
%\mathcal{L}_{\frac{\partial}{\partial r}} \theta =0.
%$$
Define a symmetric $(0,2)$ tensor field on $\Real^{2n+2}\setminus \set{0}$ by, 
\begin{align*}
\eta&=\frac{1}{r^2}\sum_{j, k=1, j\neq k}^{n+1} \left( (x_j x_k +y_j y_k) (dx_j dx_k+dy_j dy_k) +2 x_j y_k (dx_j dy_k -dy_j dx_k)\right).
\end{align*}
One readily computes that
$$
g_{\Real}=dr^2+r^2\left( \theta^2+\eta\right) \mbox{ and } \omega_{\Real}= \frac{1}{2}r^2d\theta + r dr \wedge \theta.
$$
Hence, for $X,Y\in T_p \Real^{2n+2}$, 
\begin{equation} \label{LeviFormEqn}
-\frac{1}{2} d\theta(X, J_{\Real}(Y))= -\frac{1}{r^2}\omega_{\Real}(X, J_{\Real}(Y))+ \frac{1}{r} (dr \wedge \theta)(X, J_{\Real}(Y))= \eta(X,Y).
\end{equation}
It is convenient to introduce vector fields
$$
\mathbf{X} =\sum_{j=1}^{n+1} \left(x_j  \frac{\partial}{\partial x_j}+y_j  \frac{\partial}{\partial y_j}\right) \mbox{ and }\mathbf{T}= \sum_{j=1}^{n+1} \left(x_j  \frac{\partial}{\partial y_j}-y_j  \frac{\partial}{\partial x_j}\right) =-J_{\Real}(\mathbf{X})
$$
where $\mathbf{X}$ is the position vector and $\mathbf{T}$ is a Killing vector field. 
They satisfy
$$
dr(X)= \frac{1}{r}g_{\Real}(X, \mathbf{X}) \mbox{ and } \theta(X)=\frac{1}{r^2} g_{\Real}(X, \mathbf{T}).
$$

If $X$ is any vector field on $\Real^{2n+2}$ and $\nabla^{\Real}$ is the Levi-Civita connection of $g_{\Real}$, then, because $J_{\Real}$ is $\nabla^\Real$ parallel, 
$$
\nabla_X^{\Real} \mathbf{X}= X \mbox{ and }\nabla_{X}^\Real \mathbf{T}=\nabla_{X}^\Real \left(-J_{\Real}(\mathbf{X})\right)=-J_{\Real}(\nabla_X^\Real \mathbf{X})=-J_{\Real}(X).
$$
\subsection{Contact, Sasaki and CR-geometry of $\mathbb{S}^{2n+1}$}
By thinking of $\mathbb{S}^{2n+1}$ as the boundary of the ball $\mathbb{B}^{2n+2}$ we may endow it with a natural Sasaki structure.  This comes along with associated CR and contact structures. 

Let $\hat{\theta}$ be the pullback of $\theta$ to $\mathbb{S}^{2n+1}$. It is a contact form on $\mathbb{S}^{2n+1}$ with Reeb vector field $\hat{\mathbf{T}}$, the restriction of the tangential vector field $\mathbf{T}$.  Denote by $\mathcal{H}\subset T\mathbb{S}^{2n+1}$ the contact distribution associated to $\hat{\theta}$, that is, the vector bundle over $\mathbb{S}^{2n+1}$ satisfying,  for each $p\in \mathbb{S}^{2n+1}$, 
\begin{equation}
\mathcal{H}_p=\ker \hat{\theta}_p= \set{X\in T_p \mathbb{S}^{2n+1}: \hat{\theta}_p(X)=0}\subset T_p \mathbb{S}^{2n+1}. \label{Hp}
\end{equation}

Let $g_{\mathbb{S}}$ denote the round metric on $\mathbb{S}^{2n+1}$ induced from $g_{\Real}$.  It is clear that $\mathcal{H}$ is $g_{\mathbb{S}}$ orthogonal to $\hat{\mathbf{T}}$.  It follows that $J_{\mathcal{H}}=J_{\Real}|_\mathcal{H}$ is a bundle automorphism of $\mathcal{H}$.  We may extend this to a bundle map $J_{\mathbb{S}}: T\mathbb{S}^{2n+1}\to T\mathbb{S}^{2n+1}$ by setting $J_{\mathbb{S}}(\hat{\mathbf{T}})=\mathbf{0}$.  

The triple $(J_{\mathbb{S}}, \hat{\mathbf{T}}, \hat{\theta})$ is, in the language of \cite{BlairBookNew}, an \emph{almost contact structure} on $\mathbb{S}^{2n+1}$.  In fact, together with the metric $g_{\mathbb{S}}$ this is an \emph{almost contact metric structure} and this almost contact metric structure is also \emph{Sasakian}. Indeed, by \cite[Theorem 6.3]{BlairBookNew}, if  $\nabla^{\mathbb{S}}$ is the Levi-Civita connection of $g_{\mathbb{S}}$, then it suffices to check
$$
(\nabla^{\mathbb{S}}_X  J_{\mathbb{S}})  Y= g_{\mathbb{S}}(X,Y) \hat{\mathbf{T}} -\hat{\theta}(Y) X,
$$
holds for $X,Y\in T_p \mathbb{S}^{n+1}$.
This follows from $\nabla^\Real_X J_{\Real}=0$. Hence, for $X\in T_p \mathbb{S}^{2n+1}$, 
$$
\nabla_X^{\mathbb{S}} \hat{\mathbf{T}}= -J_{\mathbb{S}}(X).
$$

In a similar vein, by using the identification $\mathbb{S}^{2n+1}\simeq \partial \mathbb{B}_{\mathbb{C}}^{n+1}$ where 
$$
\mathbb{B}^{n+1}_{\mathbb{C}}=\set{(z_1, \ldots, z_{n+1}): |z_1|^2+\ldots +|z_{n+1}|^2<1}\subset \mathbb{C}^{n+1},
$$
is the unit complex ball, 
one may interpret $(\mathcal{H}, J_{\mathbb{S}})$ as a CR-structure.  In this case, $\hat{\theta}$ is a  pseudo-convex pseudo-hermitian form, as the \emph{Levi form}, $L_{\hat{\theta}}$, is a positive definite inner product on $\mathcal{H}$ -- see \cite{dragomirDifferentialGeometryAnalysis2006}.  

Given a $C^1$ function defined on $\mathbb{S}^{2n+1}$ we denote
$$
\nabla^{\mathcal{H}} f= \nabla^{\mathbb{S}} f - g_{\mathbb{S}} (\nabla^{\mathbb{S}}f, \hat{\mathbf{T}})\hat{\mathbf{T}}
$$
where $\mathcal{H}$ is the tangential component of the gradient.

\subsection{Complex automorphisms of unit ball in $\mathbb{C}^{n+1}$ and CR-automorphisms of $\mathbb{S}^{2n+1}$} \label{backgroundCR}
We denote by $Aut_{\mathbb{C}}(\mathbb{B}_{\mathbb{C}}^{n+1})$ the group of biholomorphic automorphisms of the unit disk.  We refer to \cite{RudinDiskBook} and \cite{GoldmanCHBook} for properties of this group, but summarize some of the needed facts.

First of all, we observe that for any $A\in \mathbf{U}(n+1)$, a unitary matrix, the map
$$
\Phi_{A}: \mathbf{z}\mapsto A\cdot \mathbf{z}
$$
is an element of $Aut_{\mathbb{C}}(\mathbb{B}_{\mathbb{C}}^{n+1})$.  Secondly, for any fixed $\mathbf{b}\in \mathbb{B}^{n+1}_{\mathbb{C}}$ there is an element $\Phi_{\mathbf{b}}\in Aut_{\mathbb{C}}(\mathbb{B}_{\mathbb{C}}^{n+1})$ given by
\begin{align*}
\Phi_{\mathbf{b}}:\mathbf{z}\mapsto \sqrt{1-|\mathbf{b}|^2} \frac{\mathbf{z}}{1+ \bar{\mathbf{b}}\cdot \mathbf{z}}+\frac{1}{1+\sqrt{1-|\mathbf{b}|^2}}\left(1+\frac{\sqrt{1-|\mathbf{b}|^2}}{1+\bar{\mathbf{b}}\cdot \mathbf{z}} \right) \mathbf{b}.
\end{align*}
This map can also be expressed as
\begin{align*}
\Phi_{\mathbf{b}}(\mathbf{z})	&=\sqrt{1-|\mathbf{b}|^2} \frac{\mathbf{z}+\mathbf{b}}{1+ \bar{\mathbf{b}}\cdot \mathbf{z}}+\frac{1}{1+\sqrt{1-|\mathbf{b}|^2}}\frac{|\mathbf{b}|^2+\bar{\mathbf{b}}\cdot {\mathbf{z}}}{1+\bar{\mathbf{b}}\cdot \mathbf{z}}  \mathbf{b}.
\end{align*}
Using the identification $\mathbb{B}_{\mathbb{C}}^{n+1} \simeq \mathbb{B}^{2n+2}$, we may also think  of $\Phi_{\mathbf{b}}$ as a self-diffeomorphism of $\mathbb{B}^{2n+2}$ that is a $J_{\Real}$-holomorphic automorphism of $\mathbb{B}^{2n+2}$, in the sense that
$$
J_{\Real}\circ D_p \Phi =D_p\Phi\circ J_{\Real}.
$$
We identify  $Aut_{\mathbb{C}}(\mathbb{B}_{\mathbb{C}}^{n+1})$ with $Aut_{J}(\mathbb{B}^{2n+2})$, the group of $J_{\Real}$-holomorphic automorphisms of the ball.

Every element $\Phi\in Aut_{\mathbb{C}}(\mathbb{B}_{\mathbb{C}}^{n+1})$ satisfies 
\begin{equation}\label{StructureAutCEqn}
\Phi=\Phi_A\circ \Phi_{\mathbf{b}}
\end{equation}
for some $A\in \mathbf{U}(n+1)$ and $\mathbf{b}\in \mathbb{B}_{\mathbb{C}}^{n+1}$. 
It follows from \eqref{StructureAutCEqn} and \cite[Proposition 2.1]{BernBhattCplxHyp} that every element $\Phi\in Aut_{J}(\mathbb{B}^{2n+2})$ extends smoothly to a map $\bar{\Phi}: \bar{\mathbb{B}}^{2n+2}\to \bar{\mathbb{B}}^{2n+2}$.  We write $Aut_{J}(\bar{\mathbb{B}}^{2n+2})$ for the group of extended maps. The maps $\bar{\Phi}\in Aut_{J}(\bar{\mathbb{B}}^{2n+2})$ have the additional property that $\Psi=\bar{\Phi}|_{\mathbb{S}^{2n+1}}$ is a diffeomorphism of $\mathbb{S}^{2n+1}$ to itself, which is a \emph{CR-automorphism} of $\mathbb{S}^{2n+1}$, that is,
$$
D_p{\Psi}(\mathcal{H}_p)=\mathcal{H}_{{\Psi}(p)}  \mbox{ and on $\mathcal{H}_p$, }
J_{\mathbb{S}}\circ D_p{\Psi}=D_p{\Psi} \circ J_{\mathbb{S}}.
$$
Let us denote the set of such maps by $Aut_{CR}(\mathbb{S}^{2n+1})$.  For $\mathbf{b}=\mathbf{b}_1+i \mathbf{b}_2\in \mathbb{B}_{\mathbb{C}}^{n+1}$, let $\Psi_{\mathbf{b}}\in Aut_{CR}(\mathbb{S}^{2n+1})$ be the restriction of $\bar{\Phi}_{\mathbf{b}}\in Aut_J(\mathbb{B}^{2n+2})$.  For $A\in \mathbf{U}(n+1)$ define $\Psi_A\in Aut_{CR}(\mathbb{S}^{2n+1})$ in the same manner.
We note that if $\mathbf{b}_0\in \partial \mathbb{B}_{\mathbb{C}}^{n+1}$, 
\begin{equation} \label{ConvergencePsibEqn}
\lim_{\mathbf{b}\to \mathbf{b}_0} \Psi_{\mathbf{b}}(\mathbf{z})= \left\{\begin{array}{cc} \mathbf{b}_0 & \mathbf{z}\in \mathbb{S}^{2n+1}\setminus \set{\mathbf{b}_0} \\ DNE & \mathbf{z}=-\mathbf{b}_0. \end{array} \right.
\end{equation}
Moreover, the convergence is uniform on compact subsets of $\mathbb{S}^{2n+1}\setminus \set{\mathbf{b}_0}$.
If $\mathbf{v}$ satisfies $\bar{\mathbf{v}}\cdot \mathbf{b}_0+\mathbf{v}\cdot \bar{\mathbf{b}}_0<0$, then for $\epsilon>0$ sufficiently small $\mathbf{b}_0+\epsilon \mathbf{v}\in \mathbb{B}_{\mathbb{C}}^{n+1}$ and
$$
\lim_{\epsilon \to 0^+} \Psi_{\mathbf{b}_0+\epsilon \mathbf{v}}(-\mathbf{b}_0)=-\frac{\mathbf{v}\cdot \bar{\mathbf{b}}_0 }{  \mathbf{b}_0\cdot\bar{\mathbf{v}}} \mathbf{b}_0.
$$
Here, the right-hand side can be any point on $\mathbb{S}^{2n+1}\cap P\setminus \set{\mathbf{b}_0}$, where $P$ is the plane spanned by $\mathbf{b}_0$ and $i\mathbf{b}_0$.

\subsection{Contact gauge transformations and their Levi-Civita connections}\label{ContactGaugeConnectionSec}
Fix $\mathbf{b}=\mathbf{b}_1+i  \mathbf{b}_2\in \mathbb{B}^{n+1}_{\mathbb{C}}$  which we also think of as $(\mathbf{b}_1, \mathbf{b}_2)\in \mathbb{B}^{2n+2}$, 
 and denote by
\begin{align*}
\mathbf{S}_{\mathbf{b}}&=	-\frac{1}{2}J_{\mathbb{S}} (\nabla^{\mathcal{H}} \log W_{\mathbf{b}}) \mbox{ and }\\
\hat{\omega}_{\mathbf{b}}(X)&=g_{\mathbb{S}}(\mathbf{S}_{\mathbf{b}}, X)=\frac{1}{2}d\log W_{\mathbf{b}}( J_{\mathbb{S}}(X)).
\end{align*}
The vector field and associated one-form derived from $J_{\mathbb{S}}$ and the weight 
\begin{align*}
W_{\mathbf{b}}&=\frac{1-|\mathbf{b}_1|^2-|\mathbf{b}_2|^2}{ (1+\mathbf{b}_1\cdot \mathbf{x}+\mathbf{b}_2\cdot \mathbf{y})^2+(\mathbf{b}_1\cdot \mathbf{y} -\mathbf{b}_2\cdot \mathbf{x})^2}\\
&=\frac{1-|\mathbf{b}|^2}{ (1+g_{\Real}(\mathbf{b}, \mathbf{X}))^2+(g_{\Real}(J_{\Real}(\mathbf{b}), \mathbf{X}))^2}.
\end{align*}
Consider the family of metrics on $\mathbb{S}^{2n+1}$
$$
g_{\mathbb{S}}^{\mathbf{b}}=W_{\mathbf{b}}\left( \hat{\eta}+ \hat{\omega}_{\mathbf{b}} \cdot \hat{\theta}+ \hat{\theta}\cdot \hat{\omega}_{\mathbf{b}}+ \left( W_{\mathbf{b}} +\frac{1}{4}|\nabla^{\mathcal{H}}\log W_{\mathbf{b}}|^2_{\mathbb{S}}\right) \hat{\theta}^2\right).
$$
As shown in  \cite[Proposition 2.2]{BernBhattCplxHyp} these metrics are precisely those satisfying
$$
g_{\mathbb{S}}^{\mathbf{b}}=\Psi_{\mathbf{b}}^* g_{\mathbb{S}}
$$
and are examples of \emph{gauge transformations of contact Riemannian structures} introduced by Tanno \cite[Section 9]{Tanno1989}.  Observe  that 
\begin{equation} \label{ConformalMetEqn}
g_{\mathbb{S}}^{\mathbf{b}}|_{\mathcal{H}_p}= W_{\mathbf{b}} g_{\mathbb{S}}|_{\mathcal{H}_p},
\end{equation}
and so the $\Psi=\Psi_A\circ \Psi_{\mathbf{b}}\in Aut_{CR}(\mathbb{S}^{2n+1})$ act conformally on the horizontal bundle. 

 Let $\nabla^{\mathbf{b}}$ denote the Levi-Civita connection associated to the metric $g_{\mathbb{S}}^{\mathbf{b}}$.       We wish to compare $\nabla^{\mathbf{b}}$ and $\nabla^{\mathbb{S}}$,  the Levi-Civita connection of $g_{\mathbb{S}}$.
First, we observe that for appropriate deformation of auxiliary data, one obtains a Sasakian almost contact metric structure compatible with $g_{\mathbb{S}}^{\mathbf{b}} $ -- we omit the proof as it is straightforward.     
\begin{lem}\label{SasakianLem} Both quadruples $( J_{\mathbb{S}}, \hat{\mathbf{T}}, \hat{\theta},g_{\mathbb{S}})$ and $( J_{\mathbb{S}}^{\mathbf{b}}, \hat{\mathbf{T}}_{\mathbf{b}}, \hat{\theta}_{\mathbf{b}},g_{\mathbb{S}}^{\mathbf{b}})$ are Sasakian where
	\begin{align*}
	\hat{\theta}_{\mathbf{b}}=W_{\mathbf{b}} \hat{\theta} &=\Psi_{\mathbf{b}}^*\hat{\theta},	\hat{\mathbf{T}}_{\mathbf{b}}=\frac{1}{W_{\mathbf{b}}}(\hat{\mathbf{T}}-\mathbf{S}_{\mathbf{b}})= (\Psi_{\mathbf{b}}^{-1})_* \hat{\mathbf{T}}   \mbox{ and }  \\
	 J_{\mathbb{S}}^{\mathbf{b}}&=J_{\mathbb{S}} +\frac{1}{2} \nabla^{\mathcal{H}}\log W_{\mathbf{b}} \otimes \hat{\theta}.
\end{align*}
	In particular, for $X\in \mathcal{H}_p$, $J_{\mathbb{S}}(X)=J_{\mathbb{S}}^{\mathbf{b}}(X)$ and for $Y \in T_p \mathbb{S}^{2n+1}$,
	$$
	\nabla^{\mathbf{b}}_Y\hat{\mathbf{T}}_{\mathbf{b}} = -J_{\mathbb{S}}^{\mathbf{b}}(Y)\mbox{ and } \nabla^{\mathbb{S}}_Y\hat{\mathbf{T}}=-J_{\mathbb{S}}(Y).
	$$ 
\end{lem}

There is a simple relationship between $\nabla^{\mathbf{b}}$ and $\nabla^{\mathbb{S}}$ for sections of $\mathcal{H}$.
\begin{prop}\label{ConnectionProp}
	Let $Y\in \Gamma( \mathcal{H})$ be a (local) smooth section of the horizontal bundle $\mathcal{H}$ defined near $p\in \mathbb{S}^{2n+1}$ and $X\in \mathcal{H}_p$ be a vector. We have
	\begin{align*}
		\nabla^{\mathbf{b}}_X Y -\nabla^{\mathbb{S}}_X Y&= \frac{1}{2}\big( (X \cdot \log W_{\mathbf{b}}) Y +(Y\cdot \log W_{\mathbf{b}} )X\big)\\
			& -g_{\mathbb{S}}(\mathbf{S}_{\mathbf{b}},Y) J_{\mathbb{S}} (X)- g_{\mathbb{S}}(\mathbf{S}_{\mathbf{b}},X) J_{\mathbb{S}}(Y)- g_{\mathbb{S}}(X,Y) J_{\mathbb{S}}(\mathbf{S}_{\mathbf{b}}).
		\end{align*}
\end{prop}
\begin{proof}
To simplify the computations we connect $g_{\mathbb{S}}=g_0$ to $g_{\mathbb{S}}^{\mathbf{b}}=g_3$ by a sequence of intermediate metrics. 
First of all, let
$$
g_1=g_0+ \omega_{\mathbf{b}} \cdot \hat{\theta}+\hat{\theta}\cdot \omega_{\mathbf{b}}= g_{\mathbb{S}}+ \omega_{\mathbf{b}} \cdot \hat{\theta}+\hat{\theta}\cdot \omega_{\mathbf{b}}
$$
which is a rank two deformation of $g_0$.  Next, let
\begin{align*}
g_2&= W_{\mathbf{b}} g_1  = W_{\mathbf{b}}g_{\mathbb{S}}+   W_{\mathbf{b}}\left(\omega_{\mathbf{b}} \cdot \hat{\theta}+  \hat{\theta}\cdot \omega_{\mathbf{b}}\right),
\end{align*}
which is a conformal deformation of $g_1$.  Finally, 
$$
g^{\mathbf{b}}_{\mathbb{S}}=g_3=  g_2+\alpha \hat{\theta}^2=g_2+\alpha W_{\mathbf{b}}^{-2}  \hat{\theta}^2_{\mathbf{b}}
$$
is a rank one deformation of $g_2$ where
$$
\alpha= W_{\mathbf{b}}-1+\frac{1}{4} |\nabla^{\mathcal{H}} \log W_{\mathbf{b}}|^2.
$$
In what follows we denote by $\nabla^j$ the Levi-Civita connections of $g_j$.

%Let $\mathbf{T}_0$ and $\mathbf{S}_0$ satisfy
%$$
%g_0(\mathbf{T}_0, Z)= \hat{\theta}(Z) \mbox{ and }g_0(\mathbf{S}_0,Z)= \hat{\omega}_{\mathbf{b}}(Z)
%$$
%that is,
%$$
%\mathbf{T}_0=\hat{\mathbf{T}} \mbox{ and } \mathbf{S}_0=-\frac{1}{2}J_{\Real} (\nabla^{\mathcal{H}} \log W_{\mathbf{b}}).
%$$
Let $\mathbf{T}_0=\hat{\mathbf{T}}$ and $\mathbf{S}_0=\mathbf{S}_{\mathbf{b}}$.  By the definition of Sasakian almost contact metric structure and Lemma \ref{SasakianLem}, 
\begin{align*}\nabla_{\mathbf{T}_0}^0 \mathbf{T}_0 =\mathbf{0}, |\mathbf{T}_0|_{g_0}&=1 \mbox{ and for $X\in \Gamma(\mathcal{H})$,} \\
\nabla_X^0 \mathbf{T}_0&=-J_{\mathbb{S}}(X).
\end{align*}
Hence, Proposition \ref{RankTwoProp} implies that, for $X,Y\in \Gamma(\mathcal{H})$,
$$
\nabla^1_X Y-\nabla^0_X Y=a^{10}(X,Y) (\mathbf{T}_0-\mathbf{S}_0)-g_0(\mathbf{S}_0,Y) J_{\mathbb{S}} (X)-g_0(\mathbf{S}_0,X) J_{\Real} (Y).
$$
Likewise, let $\mathbf{T}_3=\hat{\mathbf{T}}_{\mathbf{b}}$.  By Lemma \ref{SasakianLem},
$$
(\Psi_{\mathbf{b}})_* \mathbf{T}_0=\mathbf{T}_3,  |\mathbf{T}_3|_3^2=1, \nabla_{\mathbf{T}_3}^3 \mathbf{T}_3=-J_{\mathbb{S}}^{\mathbf{b}}(\mathbf{T}_3)=\mathbf{0}.
$$
Moreover, for $X,Y\in \Gamma(\mathcal{H})$, $X$ and $Y$ are $g_3$-orthogonal to $\mathbf{T}_3$ and satisfy
$$
\nabla_X^3\mathbf{T}_3=-J_{\mathbb{S}}^{\mathbf{b}} (X)=J_{\mathbb{S}}(X) \mbox { and }g_{3}(J_{\mathbb{S}}(X),Y)=-g_3(X, J_{\mathbb{S}}(Y)).
$$
Hence,  $g_2$ is a rank one deformation of $g_3$
that satisfies the hypotheses of Proposition \ref{RankOneProp} and so, for $X,Y\in \Gamma(\mathcal{H})$, 
\begin{align*}
\nabla_X^2 Y-\nabla_X^3 Y&=\mathbf{0}.
\end{align*}

Next, we compute that
\begin{align*}
\nabla_{g_1} f &= \frac{ \mathbf{T}_0 \cdot f - \mathbf{S}_0 \cdot f }{1-|\mathbf{S}_0|^2_0} \mathbf{T}_0 +\frac{-  \mathbf{T}_0 \cdot f + |\mathbf{S}_0|^{-2}_0\mathbf{S}_0 \cdot f}{1-|\mathbf{S}_0|^2_0} \mathbf{S_0}-\frac{1}{|\mathbf{S}_0|^2_0} (\mathbf{S}_0 \cdot f)\mathbf{S}_0+\nabla^{\mathcal{H}} f \\ 
&= \frac{\mathbf{T}_0 \cdot f -\mathbf{S}_0\cdot f}{1-|\mathbf{S}_0|^2_0}(\mathbf{T}_0-\mathbf{S}_0)+\nabla^{\mathcal{H}} f .
\end{align*}
%As 
%$$\mathbf{S}_0\cdot \log W_{\mathbf{b}}=g_0(\mathbf{S}_0, \nabla^{\mathcal{H}} \log W_{\mathbf{b}})=-\frac{1}{2} g_0(J_{\Real} \nabla^{\mathcal{H}} \log W_{\mathbf{b}}, \nabla_{\mathcal{H}} \log W_{\mathbf{b}})=0,$$
Hence,
\begin{align*}
\nabla_{g_1} \log W_{\mathbf{b}}&=\frac{(\mathbf{T}_0-\mathbf{S}_0) \cdot \log W_{\mathbf{b}}}{1-|\mathbf{S}_0|^2_0}(\mathbf{T}_0-\mathbf{S}_0)+\nabla^{\mathcal{H}} \log W_{\mathbf{b}}\\
&=\frac{(\mathbf{T}_0 -\mathbf{S}_0)\cdot \log W_{\mathbf{b}}}{1-|\mathbf{S}_0|^2_0}(\mathbf{T}_0-\mathbf{S}_0)+2J_{\mathbb{S}}(\mathbf{S}_0).
\end{align*}
The formula for the change of connection under a conformal transformation gives
\begin{align*}
	\nabla_{X}^2 Y- \nabla_{X}^1 Y &= \frac{1}{2} \left((X \cdot \log W_{\mathbf{b}}) Y +(Y\cdot \log W_{\mathbf{b}} )X-g_1(X,Y) \nabla_{g_1}\log W_{\mathbf{b}}\right).
\end{align*}
Putting everything together, for $X,Y\in \Gamma(\mathcal{H})$, one obtains
\begin{align*}
\nabla_{X}^3 Y-\nabla_X^0 Y & =a^{30}(X,Y)(\mathbf{T}_0-\mathbf{S}_{\mathbf{b}}) +\frac{1}{2}\big( (X \cdot \log W_{\mathbf{b}}) Y +(Y\cdot \log W_{\mathbf{b}} )X\big)\\
&- g_0(X,Y) J_{\mathbb{S}}(\mathbf{S}_{\mathbf{b}}) -g_0(\mathbf{S}_{\mathbf{b}},Y) J_{\mathbb{S}} (X)- g_0(\mathbf{S}_{\mathbf{b}},X) J_{\mathbb{S}}(Y),
\end{align*}
where we used $\mathbf{S}_0=\mathbf{S}_{\mathbf{b}}$ and have
$$
a^{30}(X,Y)= a^{10}(X,Y)+\frac{(\mathbf{T}_0 -\mathbf{S}_0)\cdot \log W_{\mathbf{b}}}{1-|\mathbf{S}_0|^2_0}g_0(X,Y)=\hat{\theta}( \nabla_X^3 Y-\nabla_X^0 Y).
$$
In fact, $a^{30}(X,Y)=0$.  This could be shown by direct computation but instead, we use the fact that $g_0$ and $g_3$ are Sasakian. Indeed,  for $X,Y\in \Gamma(\mathcal{H})$, 
\begin{align*}
\hat{\theta}(\nabla_{X}^0 Y) &= g_0(\nabla_{X}^0 Y, \mathbf{T}_0)=-g_0( Y, \nabla_{X}^0 \mathbf{T}_0)= g_0(X,J_{\mathbb{S}}(Y))= \hat{\eta}(X,J_{\mathbb{S}}(Y);
\end{align*}
\begin{align*}
	\hat{\theta}(\nabla_{X}^3 Y)&=W_{\mathbf{b}}^{-1}\hat{\theta}_{\mathbf{b}}(\nabla_{X}^3 Y)=W_{\mathbf{b}}^{-1}g_3(\nabla^3_X Y, \mathbf{T}_3)=-W_{\mathbf{b}}^{-1}g_3(Y, \nabla^3_X  \mathbf{T}_3)\\
	& = W_{\mathbf{b}}^{-1}g_3(Y, J_{\mathbb{S}}^{\mathbf{b}}(X))= \hat{\eta}(X,J_{\mathbb{S}}(Y)).
\end{align*}
The conclusion is immediate.
\end{proof}

\section{Horizontal and Legendrian submanifolds of $\mathbb{S}^{2n+1}$}
\subsection{Properties of horizontal submanifolds}
Suppose $\Sigma\subset \mathbb{S}^{2n+1}$ is a horizontal submanifold of dimension $m$, i.e., $T_p\Sigma\subset \mathcal{H}_p$ for all $p\in \Sigma$. Clearly, $\Sigma$ is horizontal if and only if $J_{\mathbb{S}}(T_p\Sigma)$ is $g_{\mathbb{S}}$-orthogonal to $T_p\Sigma$ at all $p\in \Sigma$.  In particular,  $m\leq n$.   There are natural $g_{\mathbb{S}}$-orthogonal splittings 
$$
\mathcal{H}_p=T_p \Sigma\oplus J_{\mathbb{S}}(T_p \Sigma)\oplus \hat{N}_p\Sigma \mbox{ and } N_p\Sigma= \Real \hat{\mathbf{T}}(p)\oplus J_{\mathbb{S}}(T_p \Sigma)\oplus \hat{N}_p\Sigma 
$$
of  the horizontal bundle and the normal bundle of $\Sigma$ in $\mathbb{S}^{2n+1}$.
 Observe, $J_{\mathbb{S}}(\hat{N}_p\Sigma)=\hat{N}_p\Sigma$.  When $m=n$, i.e.,  $\Sigma$ is Legendrian, $\hat{N}_p \Sigma=0$. 

For $p\in \Sigma$ and $V\in T_p \mathbb{S}^{2n+1}$, denote by $V^\top$  the orthogonal projection onto $T_p \Sigma$, by  $V^\perp$ the orthogonal projection onto $N_p\Sigma$, by $V^N$ the orthogonal projection onto $J_{\mathbb{S}}(T_p \Sigma)$, and by $V^{\hat{N}}$ the projection onto $\hat{N}_p\Sigma$. That is, when $V\in \mathcal{H}_p$,
$$
V= V^\top + V^N +V^{\hat{N}} \mbox{ and } V^\perp=V^N+V^{\hat{N}}.
$$ 
 In general, for  $Z\in\mathcal{H}_p$,
$$
J_{\mathbb{S}}(Z^\top)=(J_{\mathbb{S}}(Z))^N.
$$
When $\Sigma$ is Legendrian $J_{\mathbb{S}}(Z^\top)=(J_{\mathbb{S}}(Z))^\perp$.

Denote by $g_{\Sigma}$ the  metric induced on $\Sigma$ by $g_{\mathbb{S}}$ and let
$$\mathbf{A}_{\Sigma}^{\mathbb{S}}(X,Y)= (\nabla^{\mathbb{S}}_X Y)^\perp \mbox{ and } \mathbf{H}_\Sigma^{\mathbb{S}}=\mathrm{tr}_{g_\Sigma} \mathbf{A}_{\Sigma}^{\mathbb{S}} $$ 
be the second fundamental form  and mean curvature of $\Sigma$ with respect to $g_{\mathbb{S}}$ which are valued in $N\Sigma$.  We record some standard facts about these horizontal submanifolds -- we omit the proofs as they are standard.
\begin{prop}\label{CurvatureFactsProp}
	For $X,Y,Z\in T_p\Sigma$ let 
	\begin{align*}\sigma_{\Sigma}(X,Y,Z) &=g_{\mathbb{S}}(\mathbf{A}^{\mathbb{S}}_{\Sigma}(X,Y), J_{\mathbb{S}}(Z))=g_{\mathbb{S}}(\mathbf{A}_{\Sigma}^N(X,Y), J_{\mathbb{S}}(Z)) \\
	&\mbox{ and } \beta_\Sigma(X)= \mathrm{tr}_{{\Sigma}} \sigma_{\Sigma} (\cdot, \cdot, X).
\end{align*}
The following hold:
	\begin{enumerate}
		\item $g_{\mathbb{S}}( \mathbf{A}_{\Sigma}^{\mathbb{S}}(X,Y), \hat{\mathbf{T}})=0$ and  $g_{\mathbb{S}}( \mathbf{H}_\Sigma, \hat{\mathbf{T}})=0$;
		\item $\sigma_{\Sigma}$ is a totally symmetric $(0,3)$ tensor field;
		\item $\beta_\Sigma$ is a closed one-form that satisfies: $$\beta_\Sigma(X)=g_{\mathbb{S}}(\mathbf{H}_\Sigma^{\mathbb{S}}, J_{\mathbb{S}}(X))=g_{\mathbb{S}}(\mathbf{H}_\Sigma^N, J_{\mathbb{S}}(X))=-g_{\Sigma}(J_{\mathbb{S}}(\mathbf{H}_\Sigma^N), X).$$
	\end{enumerate} 
\end{prop}
\subsection{Transformation of extrinsic curvatures for horizontal submanifolds}
Let $\Sigma \subset \mathbb{S}^{2n+1}$ be a horizontal submanifold. Denote by 
$\mathbf{A}_{\Sigma}^{\mathbf{b}} \mbox{ and }\mathbf{H}_{\Sigma}^{\mathbf{b}}$
the second fundamental form and mean curvature of $\Sigma$ relative to $g^{\mathbf{b}}={\Psi}_{\mathbb{b}}^*g_{\mathbb{S}}$.  Let $g^{\mathbf{b}}_{\Sigma}$ be the pullback of the metric $g^{\mathbf{b}}_{\mathbb{S}}$ to $\Sigma$.   By \eqref{ConformalMetEqn} and the fact that $\Sigma$ is horizontal one has
$$
g_{\Sigma}^{\mathbf{b}}=W_{\mathbf{b}} g_{\Sigma}
$$
and the splitting of $\mathcal{H}_p$ is also preserved. 
\begin{prop}\label{CurvatureChangeProp}
	Let $\Sigma\subset \mathbb{S}^{2n+1}$ be an $m$-dimensional horizontal submanifold.  The second fundamental forms and mean curvatures with respect to $g_{\mathbb{S}}$ and $g_{\mathbb{S}}^{\mathbf{b}}$ satisfy
	\begin{align*}
		\mathbf{A}_{\Sigma}^{\mathbf{b}}(X,Y)&=\mathbf{A}_{\Sigma}(X,Y)-g_\Sigma(\mathbf{S}_{\mathbf{b}}^\top,Y) J_{\mathbb{S}} (X)- g_\Sigma(\mathbf{S}_{\mathbf{b}}^\top,X) J_{\mathbb{S}}(Y)\\
		&- g_\Sigma(X,Y)(J_{\mathbb{S}}(\mathbf{S}_{\mathbf{b}}))^\perp, \\
	\mathbf{H}_{\Sigma}^{\mathbf{b}}&=\frac{1}{W_{\mathbf{b}}} \mathbf{H}_{\Sigma}-\frac{2}{W_{\mathbf{b}}}J_{\mathbb{S}}(\mathbf{S}_{\mathbf{b}}^\top)- \frac{m}{W_{\mathbf{b}}}(J_{\mathbb{S}}(\mathbf{S}_{\mathbf{b}}))^\perp.
%&= \frac{1}{W_{\mathbf{b}}} \mathbf{H}_{\Sigma}-\frac{1}{W_{\mathbf{b}}}\nabla_\mathbb{S} (\log W_{\mathbf{b}})^N- \frac{m}{2W_{\mathbf{b}}}(\nabla^{\mathcal{H}} \log W_{\mathbf{b}})^\perp.
\end{align*}
		When $\Sigma$ is Legendrian this simplifies to
	$$
	\mathbf{H}_{\Sigma}^{\mathbf{b}}=\frac{1}{W_{\mathbf{b}}} \mathbf{H}_{\Sigma} - \frac{m+2}{W_{\mathbf{b}}}(J_{\mathbb{S}}(\mathbf{S}_{\mathbf{b}}))^\perp=\frac{1}{W_{\mathbf{b}}} \mathbf{H}_{\Sigma} - \frac{m+2}{W_{\mathbf{b}}}J_{\mathbb{S}}(\mathbf{S}_{\mathbf{b}}^\top).
	$$
\end{prop}
\begin{proof}
As $g_\Sigma^{\mathbf{b}}$ and $g_\Sigma^{\mathbb{b}}$ are conformal, standard computations give
$$
\nabla^{\mathbf{b}, \Sigma}_{X} Y=\nabla^\Sigma_{X} Y+\frac{1}{2} \left( X \cdot \log W_{\mathbf{b}} Y+ Y \cdot \log W_{\mathbf{b}} X-g_\Sigma(X,Y)  \nabla_{{\Sigma}} \log W_{\mathbf{b}}\right).
$$
Combining this with Proposition \ref{ConnectionProp} yields the formula for $\mathbf{A}^{\mathbf{b}}_\Sigma$.
Taking the trace and using the conformal relation between the metrics completes the proof.
\end{proof}

Define a symmetric tensor field $\mathbf{E}_\Sigma^N$ on $\Sigma$ by
$$
\mathbf{E}_\Sigma^N(X,Y)= g_{\Sigma}(X,Y) \mathbf{H}^N_{\Sigma}+\beta_\Sigma(X) J_{\Real}(Y)+\beta_\Sigma(Y) J_{\Real}(X).
$$
This tensor field has the properties that
$$
(X,Y,Z)\mapsto g_{\mathbb{S}}(\mathbf{E}_\Sigma^N(X,Y),J_{\mathbb{S}}(Z)) \mbox{ is totally symmetric and}
$$
$$
\mathrm{tr}_\Sigma \mathbf{E}_\Sigma^N=(m+2)\mathbf{H}_\Sigma^N.
$$
Using $\mathbf{E}_\Sigma^N$ we introduce a trace-free symmetric tensor field which helps measure ``CR-flatness" of $\Sigma$ in the direction of $J_{\mathbb{S}}(T\Sigma)$:
$$
{\mathbf{U}}_{\Sigma}^N(X,Y)=\mathbf{A}^N_\Sigma(X,Y)- \frac{1}{m+2}\mathbf{E}_\Sigma^N(X,Y).
$$
In the direction, $\hat{N}\Sigma$, the usual trace-free part of the second fundamental form is the relevant measure of ``CR-flatness".  That is, 
$$
\mathring{\mathbf{A}}^{\hat{N}}_\Sigma(X,Y)= \mathbf{A}_{\Sigma}^{\hat{N}}(X,Y) -\frac{1}{m} \mathbf{H}_\Sigma^{\hat{N}} g_\Sigma(X,Y)
$$
where $\mathbf{A}^{\hat{N}}_\Sigma(X,Y)$ and $\mathbf{H}^{\hat{N}}_\Sigma$ are the projection of $\mathbf{A}_\Sigma(X,Y)$ and $\mathbf{H}_\Sigma$ to $\hat{N}\Sigma$.
In particular, by construction, ${\mathbf{U}}_{\Sigma}^N$ and $\mathring{\mathbf{A}}^{\hat{N}}_\Sigma$ are orthogonal.  When ${\mathbf{U}}_\Sigma$ and $\mathring{\mathbf{A}}^{\hat{N}}_\Sigma$ both identically vanish we say that $\Sigma$ is \emph{CR-umbilical}. 
\begin{prop}\label{UInvProp}
	Let $\Sigma\subset \mathbb{S}^{2n+1}$ be a horizontal submanifold.  For $\Psi\in Aut_{CR}(\mathbb{S}^{2n+1})$ and $X,Y\in T_p \Sigma$, 
when	$\Sigma'=\Psi(\Sigma)$, one has
	$$ 
	\Psi_* (\mathbf{U}_\Sigma^N(X,Y))= \mathbf{U}_{\Sigma'}^N(\Psi_* X, \Psi_* Y) \mbox{ and } \Psi_* (\mathbf{A}_\Sigma^{\hat{N}}(X,Y))=\mathbf{A}_{\Sigma'}^{\hat{N}}(\Psi_* X, \Psi_* Y) .
	$$
\end{prop}
\begin{proof}
Factor $\Psi=\Psi_A\circ \Psi_{\mathbf{b}}$ where $A\in \mathbf{U}(n+1)$ and $\mathbf{b}\in \mathbb{B}_{\mathbb{C}}^{n+1}$.
Hence, $
\Psi^* g_{\mathbb{S}}= g^{\mathbf{b}}_{\mathbb{S}}.$	
Let us denote by $(\mathbf{U}_\Sigma^N)^{\mathbf{b}}$ and $(\mathring{\mathbf{A}}_\Sigma^{\hat{N}})^{\mathbf{b}}$ the tensor fields of $\Sigma$ defined above but using the metric $g_{\mathbb{S}}^{\mathbf{b}}$.  
It is enough to establish
$$
({\mathbf{U}}_{\Sigma}^N)^{\mathbf{b}}={\mathbf{U}}_{\Sigma}^N \mbox{ and }(\mathring{\mathbf{A}}^{\hat{N}}_\Sigma)^{\mathbf{b}}=\mathring{\mathbf{A}}^{\hat{N}}_\Sigma,
$$
both of which follow from Proposition \ref{CurvatureChangeProp}.
\end{proof}

Finally, we record a useful formula for how a third order curvature transforms.
\begin{prop}\label{divJHprop}
	Let $\Sigma\subset \mathbb{S}^{2n+1}$ be a horizontal submanifold of dimension $m$.  The following holds:
	\begin{align*}
		\mathrm{div}_{g_{\Sigma}^{\mathbf{b}}} & J_{\mathbb{S}}((\mathbf{H}_\Sigma^{\mathbf{b}})^N)=W_{\mathbf{b}}^{-1} \left(\mathrm{div}_{g_\Sigma} J_{\mathbb{S}}(\mathbf{H}_\Sigma^N)+2m \mathbf{S}_{\mathbf{b}}\cdot \mathbf{H}_\Sigma^N+(m+2)\mathbf{S}_{\mathbf{b}}\cdot \mathbf{H}_\Sigma^{\hat{N}}\right)\\
		& -\frac{m(m+2)}{4} W_{\mathbf{b}}\nabla_\Sigma  W_{\mathbf{b}}^{-1} \cdot J_{\mathbb{S}}(\nabla_{\mathbb{S}}  W_{\mathbf{b}}^{-1} )-\frac{m(m+2)}{1-|\mathbf{b}|^2} J_{\Real}(\mathbf{b})\cdot \mathbf{X}.
	\end{align*}
\end{prop}
\begin{proof}
We first record some useful calculations
	\begin{equation}
		\label{NablaWbEqn}
		\nabla_{\Real} W_{\mathbf{b}}^{-1} = \frac{ 2(1+\mathbf{b}\cdot\mathbf{X}) \mathbf{b}+ 2(J_{\Real}(\mathbf{b})\cdot \mathbf{X} ) J_{\Real}(\mathbf{b})}{1-|\mathbf{b}|^2}
	\end{equation}
	\begin{equation}
		\label{HessWbEqn}
	\nabla^2_{\Real}W_{\mathbf{b}}^{-1} =\frac{ 2 \mathbf{b}\otimes \mathbf{b}+ 2 J_{\Real}(\mathbf{b}) \otimes J_{\Real}(\mathbf{b})}{1-|\mathbf{b}|^2}.
\end{equation}
It follows from \eqref{NablaWbEqn} that
$$
\hat{\mathbf{T}} \cdot W_{\mathbf{b}}^{-1} = \frac{2}{1-|\mathbf{b}|^2} \mathbf{b}\cdot \hat{\mathbf{T}}=\frac{2}{1-|\mathbf{b}|^2} J_{\Real}(\mathbf{b})\cdot {\mathbf{X}}.
$$

	By Proposition \ref{CurvatureChangeProp}, 
	$$
	J_{\mathbb{S}}((\mathbf{H}_\Sigma^{\mathbf{b}})^N)=  W_{\mathbf{b}}^{-1} J_{\mathbb{S}}(\mathbf{H}_{\Sigma}^N)+(m+2)W_{\mathbf{b}}^{-1} \mathbf{S}_{\mathbf{b}}^\top.
	$$
	Using that $g_{\Sigma}^{\mathbf{b}}= W_{\mathbf{b}} g_{\Sigma}$ and definition of divergence, we have 
	\begin{align*} \mathrm{div}_{g_{\Sigma}^{\mathbf{b}}} J_{\mathbb{S}}((\mathbf{H}_\Sigma^{\mathbf{b}})^N)&=W_{\mathbf{b}}^{-1} \mathrm{div}_{g_\Sigma} J_{\mathbb{S}}(\mathbf{H}_\Sigma^N)+\frac{m-2}{2} W_{\mathbf{b}}^{-2} \nabla_{\Sigma} W_{\mathbf{b}} \cdot  J_{\mathbb{S}}(\mathbf{H}_\Sigma^N) \\
		&+(m+2) W_{\mathbf{b}}^{-1} \mathrm{div}_{g_\Sigma} (\mathbf{S}_{\mathbf{b}}^\top)-\frac{m^2-4}{2}  \nabla_{\Sigma} W_{\mathbf{b}}^{-1} \cdot   \mathbf{S}_{\mathbf{b}}^\top. 
	\end{align*}
It is straightforward to see
$$
\frac{m-2}{2} W_{\mathbf{b}}^{-2} \nabla_{\Sigma} W_{\mathbf{b}} \cdot  J_{\mathbb{S}}(\mathbf{H}_\Sigma^N)=(m-2) W_{\mathbf{b}}^{-1}\mathbf{S}_{\mathbf{b}} \cdot \mathbf{H}_\Sigma^N \mbox{ and }
$$
$$
-\frac{m^2-4}{2}  \nabla_{\Sigma} W_{\mathbf{b}}^{-1} \cdot   \mathbf{S}_{\mathbf{b}}^\top=-\frac{m^2-4}{4}  W_{\mathbf{b}} \nabla_{\Sigma} W_{\mathbf{b}}^{-1} \cdot J_{\mathbb{S}}(\nabla_{\mathbb{S}} W_{\mathbf{b}}^{-1}).
$$

A basic computation yields
	\begin{align*}
		\mathrm{div}_{g_\Sigma} (\mathbf{S}_{\mathbf{b}}^\top)&= \mathrm{div}_{g_\Sigma} (\mathbf{S}_{\mathbf{b}})+\mathbf{H}_\Sigma \cdot \mathbf{S}_{\mathbf{b}}=\mathrm{div}_{g_\Sigma} (\mathbf{S}_{\mathbf{b}})+ \mathbf{H}_\Sigma^N \cdot \mathbf{S}_{\mathbf{b}}^N+\mathbf{H}_\Sigma^{\hat{N}} \cdot \mathbf{S}_{\mathbf{b}}^{\hat{N}}.
	\end{align*}
While we have	
	\begin{align*}
		\mathrm{div}_\Sigma \mathbf{S}_{\mathbf{b}}&=-\frac{1}{2}\sum_{i=1}^m g_\mathbb{S}(\nabla^\mathbb{S}_{E_i} J_{\mathbb{S}}(\nabla^{\mathcal{H}} \log W_{\mathbf{b}}), E_i)=-\frac{1}{2}\sum_{i=1}^m g_\mathbb{S}(\nabla^\mathbb{S}_{E_i} J_{\mathbb{S}}(\nabla_{\mathbb{S}} \log W_{\mathbf{b}}), E_i)\\
		&=-\frac{1}{2}\sum_{i=1}^m g_\mathbb{S}( J_{\mathbb{S}}(\nabla^\mathbb{S}_{E_i}\nabla_{\mathbb{S}} \log W_{\mathbf{b}})-\hat{\theta}(\nabla_{\mathbb{S}} \log W_{\mathbf{b}})  E_i , E_i)\\
		&=\frac{m}{2} \hat{\mathbf{T}}\cdot \log W_{\mathbf{b}} +
		\frac{1}{2}  \sum_{i=1}^m \nabla^2_{\mathbb{S}} \log W_{\mathbf{b}}(E_i, J_{\mathbb{S}}(E_i))
	\end{align*}
	where $E_1, \ldots, E_m$ is an orthonormal basis of $T_p \Sigma$.  
It follows from \eqref{HessWbEqn} that
	\begin{align*}
		\mathrm{div}_{\Sigma}(\mathbf{S}_{\mathbf{b}})
		&=\frac{m}{2} \hat{\mathbf{T}}\cdot \log W_{\mathbf{b}} +\frac{1}{2} g_{\mathbb{S}}(\nabla_{\mathbb{S}} \log W_{\mathbf{b}}, J_{\mathbb{S}}(\nabla_{\Sigma} \log W_{\mathbf{b}}))\\
		&= -\frac{m}{1-|\mathbf{b}|^2} W_{\mathbf{b}} J_{\Real}(\mathbf{b})\cdot {\mathbf{X}}-\frac{1}{2}W_{\mathbf{b}}^{2} \nabla_{\Sigma} W_{\mathbf{b}}^{-1}\cdot J_{\mathbb{S}}\nabla_{\mathbb{S}}  W_{\mathbf{b}}^{-1}.
%		&=\frac{m}{2} \hat{\mathbf{T}}\cdot \log W_{\mathbf{b}} +2 g_{\mathbb{S}}( \mathbf{S}_{\mathbf{b}}^\top, \mathbf{S}_{\mathbf{b}}^N).
	\end{align*}
Putting everything together gives the result.
\end{proof}

\subsection{Characterization of CR-umbilical horizontal submanifolds} \label{UmbilicCharSec}

It is useful to have the following lemma:
\begin{lem}\label{DefRotLem}
	For any $p\in \mathbb{S}^{2n+1}$ and $\mathbf{b}\in \mathbb{B}^{2n+2}$, there is a $\Psi\in \mathrm{Aut}_{CR}(\mathbb{S}^{2n+1})$ so
	$$
	\Psi(p)=p \mbox{ and }  D\Psi_p |_{\mathcal{H}_p} =\sqrt{W_{\mathbf{b}}(p)} I|_{\mathcal{H}_p}.
	$$
%	 In particular, 
%	 $$
%	 D\Psi_p (\mathcal{H}_p)=\mathcal{H}_p.
%	 $$
\end{lem}
\begin{proof}
Let $p'=\Psi_{\mathbf{b}}(p)$.  As $\mathbf{U}(n+1)$ acts transitively on $\mathbb{S}^{2n+1}$, there is an $A\in \mathbf{U}(n+1)$ such that $\Psi_A(p')=p$.  Now let $\mathbf{U}_p(n+1)$ be the subgroup of $\mathbf{U}(n+1)$ that fixes $p$.  It is not hard to see that if $\Psi'=\Psi_A \circ \Psi_{\mathbf{b}}$, then $\Psi'(p)=p$, and by \eqref{ConformalMetEqn} 
$$
D_p\Psi'|_{\mathcal{H}_p} =\sqrt{W_{\mathbf{b}}(p)}  D_p \Psi_{B}|_{\mathcal{H}_p}
$$
for $B\in \mathbf{U}_p(n+1)$.  Hence, $\Psi= \Psi_{B}^{-1}\circ \Psi'$, is the desired element.
\end{proof}

\begin{prop}\label{NormalizeMCProp}
	Let $\Sigma\subset \mathbb{S}^{2n+1}$ be a horizontal submanifold.  For any $p\in \Sigma$, there exists $\Psi\in \mathrm{Aut}_{CR}(\mathbb{S}^{2n+1})$ such that if $\Sigma'=\Psi(\Sigma)$, then
	\begin{enumerate}
		\item $p\in \Sigma'$ and $T_p \Sigma=T_p\Sigma'$;
		\item $\mathbf{H}_{\Sigma'}(p)=\mathbf{0}$;
		\item $\mathrm{div}_{\Sigma'}(J_{\mathbb{S}}(\mathbf{H}^N_{\Sigma'}))(p)=0$.		
	\end{enumerate}
\end{prop}
\begin{proof}
  By Lemma \ref{DefRotLem}, the first condition can always be ensured for whatever choice of $\mathbf{b}\in \mathbb{B}^{2n+2}$ we make to ensure the other conditions; denote by $\Sigma_{\mathbf{b}}=\Psi(\Sigma)$ the corresponding surface so $p\in \Sigma_{\mathbf{b}}$ and $T_p \Sigma_{\mathbf{b}}=T_p \Sigma$.   We also simplify by observing that, by acting by $\mathbf{U}(n+1)$, we may assume,  without loss of generality, that 
  $$
  \mathbf{X}(p)=\mathbf{e}_1 \mbox{ and } \mathbf{T}(p)=\mathbf{e}_{n+2},
  $$
 and $T_p \Sigma$ is spanned by  $\mathbf{e}_{j+1}$ and $J_{\mathbb{S}}(T_p\Sigma)$ is spanned by $\mathbf{e}_{j+n+2}$ for  $1\leq j \leq m$. 
  
%  Using Proposition \ref{CurvatureChangeProp}, one computes that
%  $$
%  \mathbf{H}_{\Sigma}^{\mathbf{b}}(p)=\frac{1}{W_{\mathbf{b}}(p)} \mathbf{H}_{\Sigma}(p) - \frac{m}{W_{\mathbf{b}}(p)}(J_{\mathbb{S}}(\mathbf{S}_0))^\perp-\frac{2}{W_{\mathbf{b}}(p)}(J_{\mathbb{S}}(\mathbf{S}_0))^{N}.
%  	$$
  We proceed in two stages to choose $\mathbf{b}$ so the second two conditions hold.	
 First,  choose $\mathbf{b}_1$ so that $\mathbf{H}_{\Sigma_{\mathbf{b}_1}}^{\hat{N}}=\mathbf{0}$. Our simplifications ensure that
 $$
 J_{\mathbb{S}}(\mathbf{S}_{\mathbf{b}}(p))=\frac{1}{2} \nabla^{\mathcal{H}}  \log W_{\mathbf{b}}(p)=- \frac{(1+b_1)\mathbf{b}^{\mathcal{H}}(p)+ b_{n+2}J_{\Real}(\mathbf{b})^{\mathcal{H}}(p) }{(1+b_1)^2+b_{n+2}^2}
 $$ 	
 where $\mathbf{v}^{\mathcal{H}}(p)$ is the orthogonal projection of $\mathbf{v}$ onto $\mathcal{H}_p$.
  By Proposition \ref{CurvatureChangeProp}, 
in order  to find $\mathbf{b}_1$ we must ensure
$$
(J_{\mathbb{S}}(\mathbf{S}_{\mathbf{b}_1})(p))^{\hat{N}}=\frac{1}{m}  \mathbf{H}_{\Sigma}^{\hat{N}}(p).
$$
Let $\mathbf{v}_1=\mathbf{H}_\Sigma^{\hat{N}}(p)\in \hat{N}_p\Sigma\subset \mathcal{H}_p$.  Notice that $J_{\mathbb{S}}(\mathbf{v}_1)\in \hat{N}_p\Sigma$.  Set $\mathbf{b}_1=\alpha_1 \mathbf{e}_1+\beta_1 \mathbf{v}_1$ where $\alpha_1$ and $\beta_1$ are to be determined but are small enough such that $\mathbf{b}_1\in \mathbb{B}^{2n+2}$.  We observe that our normalization ensures $$\mathbf{e}_1^{\mathcal{H}}(p)=\mathbf{X}^{\mathcal{H}}=\mathbf{0}\mbox{ and } (J_{\Real}(\mathbf{e}_1))^{\mathcal{H}}(p)= \mathbf{e}_{n+2}^{\mathcal{H}}(p)=\mathbf{T}^{\mathcal{H}}(p)=\mathbf{0}.
$$
That is, $\mathbf{b}_1^{\mathcal{H}}(p)=\beta_1 \mathbf{v}_1^{\mathcal{H}}(p)=\beta_1\mathbf{H}_\Sigma^{\hat{N}}(p)$ and so
$$
J_{\mathbb{S}}(\mathbf{S}_{\mathbf{b}_1}(p))= -\frac{1}{2} \frac{\beta_1 \mathbf{H}_\Sigma^{\hat{N}}(p)}{1+ \alpha_1}.
$$
Hence, it suffices to find $\alpha_1$ and $\beta_1$ such that
$$
0=1+\frac{m\beta_1}{2(1+\alpha_1)} \mbox{ and } |\mathbf{b}_1|^2=\alpha_1^2+\beta_1^2 |\mathbf{H}_\Sigma^{\hat{N}}|^2<1.
$$
One verifies that
$$
\alpha_1= -\frac{4|\mathbf{H}_{\Sigma}^{\hat{N}}|^2}{m^2+4 |\mathbf{H}_{\Sigma}^{\hat{N}}|^2} \mbox{ and } \beta_1=-\frac{2}{m} (1+\alpha_1)=- \frac{2m}{m^2+4 |\mathbf{H}_{\Sigma}^{\hat{N}}|^2}
$$
satisfies these conditions.  Hence, we may now assume that $\Sigma$ satisfies $\mathbf{H}_{\Sigma}^{\hat{N}}(p)=\mathbf{0}$.

Next, we replace $\Sigma$ by $\Sigma_{\mathbf{b}_1}$ and find $\mathbf{b}_2$ so that $\Sigma_{\mathbf{b}_2}$ has vanishing mean curvature vector.  Set $\mathbf{v}_2=\mathbf{H}_{\Sigma}(p)=\mathbf{H}^{N}_\Sigma(p)$ -- here we use that $\mathbf{H}_\Sigma^{\hat{N}}(p)=\mathbf{0}$.  Let $\mathbf{b}_2= \alpha_2 \mathbf{e}_1+\beta_2\mathbf{v}_2\in \mathbb{B}^{2n+2}$
where  $\alpha_2$ and $\beta_2$ are to be determined. The choice of $\mathbf{b}_2$ ensures $J_{\mathbb{S}}(\mathbf{S}_{\mathbf{b}_2})^{\hat{N}}=\mathbf{0}$ and so  Proposition \ref{CurvatureChangeProp} gives 
$$
\mathbf{H}_{\Sigma}^{{\mathbf{b}_2}}(p)= \frac{1}{W_{\mathbf{b}_2}(p)} \mathbf{H}_\Sigma(p)\left( 1+\frac{m+2}{2} \frac{\beta_2}{1+\alpha_2}\right).
$$
As above, we may  take
$$
\alpha_2= -\frac{4|\mathbf{H}_\Sigma|^2}{(m+2)^2+4 |\mathbf{H}_\Sigma|^2} \mbox{ and } \beta_2=-\frac{2}{m+2}(1+\alpha_2)= -\frac{2(m+2)}{(m+2)^2+4|\mathbf{H}_\Sigma|^2}
$$
 which ensures a $\Sigma$ that satisfies the first two conditions.

To conclude, set $\mathbf{b}_3=b_1 \mathbf{e}_1+b_{n+2} \mathbf{e}_{n+2}$. The hypotheses on $\Sigma$ ensure
$$
\mathbf{b}^\top_3(p)=(J_{\mathbb{S}}(\mathbf{b}_3))^\top(p)=\mathbf{S}_{\mathbf{b}_3}^\top(p)=\mathbf{0}
$$
and so $\mathbf{H}_{\Sigma_{\mathbf{b}_3}}(p)=\mathbf{H}_{\Sigma}^{\mathbf{b}_3}(p)=\nabla_\Sigma  W_{\mathbf{b}_3}^{-1}(p)=\mathbf{0}$.  Moreover, 
by Proposition \ref{divJHprop},
$$
\mathrm{div}_{\Sigma_{\mathbf{b}_3}}(J_{\mathbb{S}}(\mathbf{H}_\Sigma^{\mathbf{b}_3})^N)(p)=\frac{1}{W_{\mathbf{b}_3}(p)}\left(\mathrm{div}_\Sigma(J_{\mathbb{S}}(\mathbf{H}^N_{\Sigma}))(p)-\frac{m(m+2)b_{n+2}}{ (1+b_1)^2+b_{n+2}^2}\right).
$$
If we set $\gamma= \mathrm{div}_\Sigma(J_{\mathbb{S}}(\mathbf{H}^N_{\Sigma}))(p)$, then we wish to choose $b_1, b_{n+2}$ such that
$$
0=\gamma-\frac{m(m+2) b_{n+2}}{(1+b_1)^2+b_{n+2}^2} \mbox{ and } |\mathbf{b}_3|^2= b_1^2+b_{n+2}^2<1.
$$
This can be achieved by taking
$$
b_1=-\frac{\gamma^2}{\gamma^2+m^2(m+2)^2} \mbox{ and } b_{n+2}= \frac{m (m+2) \gamma}{\gamma^2+m^2(m+2)^2}.$$ 
Hence, we have shown all three conditions hold for $\Sigma'=\Sigma_{\mathbf{b}_3}$.
\end{proof}

It will be helpful to have the following notation: Let $\Sigma\subset \mathbb{S}^{2n+1}$ be an $m$-dimensional submanifold and $\mathbf{B},  \mathbf{C}$ be symmetric $(0,2)$-tensor fields on $\Sigma$  valued in $T\mathbb{S}^{n+1}$.   Define  a $(0,2)$ tensor field,  $\mathbf{B}\odot \mathbf{C}$, on $\Sigma$  by 
$$(\mathbf{B}\odot\mathbf{C})_p(X,Y)= \sum_{i=1}^m g_{\mathbb{S}}(\mathbf{B}_p(X, E_i), \mathbf{C}_p(E_i,Y))
$$
where $E_1, \ldots, E_m$ is an orthonormal basis of $T_p \Sigma$.  Observe that
 $$(\mathbf{B}\odot \mathbf{C})(X,Y)=(\mathbf{C}\odot \mathbf{B})(Y,X) .
 $$
Hence, $\mathbf{B}\odot \mathbf{B}$ is symmetric and $|\mathbf{B}|^2= \mathrm{tr}_\Sigma (\mathbf{B} \odot \mathbf{B})$.
For $X,Y\in T_p \Sigma$ set
$$
\mathrm{Ric}_\Sigma^N(X,Y)=g_{\mathbb{S}}(\mathbf{H}_\Sigma^N, \mathbf{A}_{\Sigma}^N(X,Y)) -(\mathbf{A}_\Sigma^N\odot \mathbf{A}_\Sigma^N)(X,Y) \mbox{ and }
$$
$$
\mathrm{Ric}_\Sigma^{\hat{N}}(X,Y)=g_{\mathbb{S}}(\mathbf{H}_\Sigma^{\hat{N}}, \mathbf{A}_{\Sigma}^{\hat{N}}(X,Y)) -(\mathbf{A}_\Sigma^{\hat{N}}\odot \mathbf{A}_\Sigma^{\hat{N}})(X,Y).
$$
By the Gauss equations, 
\begin{equation}\label{GaussEqnUeqn}
\mathrm{Ric}_\Sigma(X,Y)=\mathrm{Ric}_\Sigma^N(X,Y)+\mathrm{Ric}_\Sigma^{\hat{N}}(X,Y)+(m-1) g_\Sigma(X,Y).
\end{equation}
\begin{lem}\label{RicciULem}
If $\Sigma \subset \mathbb{S}^{2n+1}$ is an $m$-dimensional horizontal submanifold, then

	\begin{align*}
	\mathrm{Ric}_\Sigma^N&=\frac{m-2}{m+2}  \mathbf{H}_{\Sigma}^N\cdot \mathbf{U}_\Sigma^N-\mathbf{U}_\Sigma^N\odot \mathbf{U}_\Sigma^N
+\frac{m|\mathbf{H}_\Sigma^N|^2}{(m+2)^2}   g_{\Sigma}+ \frac{m-2}{(m+2)^2} \beta_\Sigma\otimes\beta_\Sigma;
\end{align*}
\begin{align*}
\mathrm{Ric}_\Sigma^{\hat{N}} =	\frac{m-2}{m}\mathbf{H}_\Sigma^{\hat{N}}\cdot \mathring{\mathbf{A}}_\Sigma^{\hat{N}} - \mathring{\mathbf{A}}_\Sigma^{\hat{N}}\odot \mathring{\mathbf{A}}_\Sigma^{\hat{N}}+\frac{m-1}{m^2}|\mathbf{H}_\Sigma^{\hat{N}}|^2  g_{\Sigma}.
\end{align*}
\end{lem}
\begin{proof}
	Let $E_{i}, 1\leq i\leq m$ be an orthonormal basis of $T_p \Sigma$.  Using the symmetries of $\sigma_\Sigma$, i.e.,  Proposition \ref{CurvatureFactsProp}, we compute that
	\begin{align*}
		\mathbf{A}_{\Sigma}^N&\odot \mathbf{E}_{\Sigma}^N=	\mathbf{E}_{\Sigma}^N\odot \mathbf{A}_{\Sigma}^N\\
		&=\sum_{i=1}^m\big( \mathbf{H}_{\Sigma}^N\cdot \mathbf{A}_\Sigma^N(\cdot ,E_i) g_\Sigma(E_i,\cdot ) + \sigma_{\Sigma}(\cdot, E_i,  \cdot) \beta_\Sigma(E_i)+\sigma_\Sigma (\cdot, E_i, E_i) \beta_\Sigma(\cdot)\big)\\
			&= 2\mathbf{H}_{\Sigma}^N\cdot \mathbf{A}_\Sigma +\beta_\Sigma\otimes\beta_\Sigma.
	\end{align*}
%	Hence,
%	\begin{align*}
%		\langle \mathbf{A}_\Sigma, \mathbf{E}_\Sigma^N\rangle_{g_\Sigma}= \langle \mathbf{A}_\Sigma^N, \mathbf{E}_\Sigma^N\rangle_{g_\Sigma}= \mathrm{tr}_\Sigma	(\mathbf{A}_{\Sigma}\odot \mathbf{E}_{\Sigma}^N)= 3|\mathbf{H}^N_{\Sigma}|^2.
%	\end{align*}
	One also has
	\begin{align*}
			\mathbf{E}_{\Sigma}^N\odot \mathbf{E}_{\Sigma}^N=(m+6)\beta_{\Sigma}\otimes\beta_{\Sigma} +2|\mathbf{H}_\Sigma^N|^2 g_\Sigma.
	\end{align*}
%	It follows that
%	$$
%	|\mathbf{E}_\Sigma^N|_{g_\Sigma}^2=\mathrm{tr}_\Sigma (\mathbf{E}_{\Sigma}^N\odot \mathbf{E}_{\Sigma}^N)= 3(2+m)|\mathbf{H}_\Sigma^N|^2.
%	$$
Hence,
	\begin{align*}
	\mathbf{U}_\Sigma^N\odot\mathbf{U}_\Sigma^N &=  \mathbf{A}_\Sigma^N\odot\mathbf{A}_\Sigma^N-\frac{2}{m+2} \mathbf{A}_\Sigma^N\odot \mathbf{E}_\Sigma^N+\frac{1}{(m+2)^2} \mathbf{E}_\Sigma^N \odot \mathbf{E}_\Sigma^N\\
		&=\mathbf{A}_\Sigma^N\odot \mathbf{A}^N_\Sigma-\frac{4}{m+2} \mathbf{H}_{\Sigma}^N\cdot \mathbf{A}_\Sigma^N+\frac{2|\mathbf{H}_\Sigma^N|^2}{(m+2)^2} g_\Sigma-\frac{m-2}{(m+2)^2} \beta_\Sigma\otimes \beta_\Sigma.
	\end{align*}
It follows that
\begin{align*}
 \mathrm{Ric}_\Sigma^N& =\frac{m-2}{m+2} \mathbf{H}_{\Sigma}^N\cdot \mathbf{A}_\Sigma+\frac{2|\mathbf{H}_\Sigma^N|^2}{(m+2)^2} g_\Sigma-\frac{m-2}{(m+2)^2} \beta_\Sigma\otimes \beta_\Sigma-\mathbf{U}_\Sigma^N\odot \mathbf{U}^N_\Sigma,
 %&= \frac{m-2}{m+2} \mathbf{H}_{\Sigma}^N\cdot \mathbf{U}_\Sigma^N(X,Y)+\frac{m}{(m+2)^2}|\mathbf{H}_\Sigma^N|^2 g_\Sigma(X,Y)\\
% &+\frac{m-2}{(m+2)^2} \beta_\Sigma(X)\beta_\Sigma(Y)-(\mathbf{U}_\Sigma^N\odot \mathbf{U}_\Sigma^N)(X,Y), 
\end{align*}
	which yields the first formula.  The second formula follows from 
	$$
	\mathbf{A}_\Sigma^{\hat{N}}\odot\mathbf{A}_\Sigma^{\hat{N}}=\mathring{\mathbf{A}}_\Sigma^{\hat{N}}\odot \mathring{\mathbf{A}}_\Sigma^{\hat{N}}+\frac{2}{m} \mathbf{H}_\Sigma^{\hat{N}} 
\cdot \mathring{\mathbf{A}}_\Sigma^{\hat{N}}+\frac{1}{m^2} |\mathbf{H}_\Sigma^{\hat{N}}|^2 g_\Sigma.
$$
\end{proof}

\begin{thm}\label{CRinv}
	Let $\Sigma\subset \mathbb{S}^{2n+1}$ be a connected horizontal submanifold of dimension $m\geq 2$. If $\mathbf{U}_{\Sigma}^N$ and $\mathring{\mathbf{A}}_\Sigma^{\hat{N}}$ both vanish identically, then $\Sigma$ is a contact Whitney sphere, i.e., there is a $\mathbf{b}\in \mathbb{B}^{2n+2}$ so that ${\Psi}_{\mathbf{b}}(\Sigma)$ is a subset of a totally geodesic sphere. 
\end{thm}
\begin{rem}
 This and related results have appeared in the literature in various forms.
  For instance, Ros-Urbano \cite{rosLagrangianSubmanifoldsn1998} characterized Whitney spheres as the Lagrangian submanifolds of $\mathbb{C}^n$ that are umbilic in an appropriate sense.  This result combined with the correspondence of Reckziegel \cite{reckziegelCorrespondenceHorizontalSubmanifolds1988} was used by Blair-Carriazo \cite{blairContactWhitneySphere} to establish an analogous result in $\Real^{2n+1}$ with its usual contact structure.    We refer also to \cite{huOptimalInequalityRelated2020}.   For the sake of completeness, we provide a self-contained proof that highlights some features that may be of independent interest.
\end{rem}

\begin{proof}
We first show that after the application of some element of $Aut_{CR}(\mathbb{S}^{2n+1})$,  $\mathbf{H}_{\Sigma}^{\hat{N}}=\mathbf{0}$.   To that end, we observe that there is a $(1,2)$ tensor field on $\Sigma$, $\mathbf{a}_\Sigma$, so that
$$
\mathbf{A}^N_\Sigma(X,Y)=J_{\mathbb{S}}(\mathbf{a}_\Sigma(X, Y))\mbox{, i.e., } \sigma_\Sigma(X,Y,Z)= g_{\mathbb{S}}(\mathbf{a}_\Sigma(X, Y), Z).
$$
For $Z$ tangent to $\Sigma$, the orthogonal splitting of the normal bundle yields
$$
(\nabla^\perp_{Z} \mathbf{A}^N_\Sigma(X,Y))^{\hat{N}}=J_{\mathbb{S}}( (\nabla^{\mathbb{S}}_Z \mathbf{a}_\Sigma(X,Y))^{\hat{N}}= J_{\mathbb{S}}( \mathbf{A}^{\hat{N}}_\Sigma(Z, \mathbf{a}_\Sigma(X,Y))).
$$
Hence, by the Codazzi equations 
\begin{align*}
(\nabla^\perp_Z &\mathbf{A}_\Sigma^{\hat{N}}(X,Y))^{\hat{N}}=
(\nabla^\perp_Z \mathbf{A}_\Sigma(X,Y)-\nabla^\perp_Z\mathbf{A}_\Sigma^{{N}}(X,Y))^{\hat{N}}\\
&=(\nabla^\perp_X \mathbf{A}_\Sigma^{\hat{N}}(Z,Y))^{\hat{N}}+J_{\mathbb{S}}( \mathbf{A}^{\hat{N}}_\Sigma(X, \mathbf{a}_\Sigma(Z,Y)))-J_{\mathbb{S}}( \mathbf{A}^{\hat{N}}_\Sigma(Z, \mathbf{a}_\Sigma(X,Y))).
\end{align*}
When $\mathring{\mathbf{A}}_\Sigma^{\hat{N}}=\mathbf{0}$ this simplifies to
$$
(\nabla^\perp_Z \mathbf{A}_\Sigma^{\hat{N}}(X,Y))^{\hat{N}}-(\nabla^\perp_X \mathbf{A}_\Sigma^{\hat{N}}(Z,Y))^{\hat{N}}=\frac{1}{m} J_{\mathbb{S}}(\mathbf{H}_\Sigma^{\hat{N}}) (\sigma_\Sigma(Z,Y, X)-\sigma_\Sigma(X, Y, Z)).
$$
The symmetry of $\sigma_\Sigma$ and the vanishing of $\mathring{\mathbf{A}}_\Sigma^{\hat{N}}$ implies, by taking the trace of $(\nabla^\perp \mathbf{A}_\Sigma^{\hat{N}})^{\hat{N}}$ and appropriate manipulation that
$$
(\nabla_X^{\perp} \mathbf{H}_\Sigma^{\hat{N}})^{\hat{N}}=\mathbf{0}.
$$ 
Thus, $X\cdot |\mathbf{H}_\Sigma^{\hat{N}}|^2=0$.  As $\Sigma$ is connected,  it suffices to ensure $\mathbf{H}_\Sigma^{\hat{N}}$ vanishes at a point. By Proposition \ref{NormalizeMCProp}, this can be arranged by applying an element of $Aut_{CR}(\mathbb{S}^{2n+1})$ and so in what follows we suppose that $\mathbf{H}^{\hat{N}}_\Sigma$ and, thus $\mathbf{A}^{\hat{N}}_\Sigma$, identically vanish.

As $\mathbf{U}_\Sigma^N$ vanishes,
$$
(m+2)\sigma_\Sigma(X,Y,Z)= g_\Sigma(X,Y) \beta_\Sigma(Z)+ g_\Sigma(Y,Z)\beta_\Sigma (X)+ g_\Sigma(X,Z)\beta_\Sigma(Y).$$
Using the Codazzi equations and vanishing of $\mathbf{A}_\Sigma^{\hat{N}}$ it follows that
\begin{align*}
(m+2) & g_\Sigma((\nabla_Y^\perp \mathbf{A}_\Sigma)(X,X), J_{\Real}(Z))=  g_\Sigma(X,Y) (\nabla_X^\Sigma\beta_\Sigma)(Z)\\
&+ g_\Sigma(Y,Z)(\nabla^\Sigma_X \beta_\Sigma) (X)+ g_\Sigma(X,Z)(\nabla_X^\Sigma \beta_\Sigma)(Y).
\end{align*}
Taking the trace and manipulating the result yields
%$$
%(m+2) g_\Sigma((\nabla_Y^\perp \mathbf{H}_\Sigma), J_{\Real}(Z))= (\nabla_Y^\Sigma\beta_\Sigma)(Z)+ g_\Sigma(Y,Z) \sum_{i=1}^m(\nabla^\Sigma_{E_i} \beta_\Sigma) (E_i)+ (\nabla_Z^\Sigma \beta_\Sigma)(Y).
%$$
%This reduces to
$$
(m+2) \nabla_Y \beta_\Sigma(Z)= (\nabla_Y^\Sigma\beta_\Sigma)(Z)+ g_\Sigma(Y,Z) \sum_{i=1}^m(\nabla^\Sigma_{E_i} \beta_\Sigma) (E_i)+ (\nabla_Z^\Sigma \beta_\Sigma)(Y).
$$
As $\beta_\Sigma$ is closed, $\nabla^\Sigma \beta_\Sigma$ is symmetric and so this equation reduces to
$$
m\nabla_Y \beta_\Sigma(Z)=  g_\Sigma(Y,Z)\sum_{i=1}^m(\nabla^\Sigma_{E_i} \beta_\Sigma) (E_i).
$$
Moreover, we can locally write $\beta_\Sigma=d f$ and so
\begin{equation}\label{ConfHessEqn}
m\nabla^2_\Sigma f(Y,Z)= (\Delta_\Sigma f )g_{\Sigma}(Y,Z).
\end{equation}
That is, the Hessian of $f$ is conformal.  Observe that
$$
\Delta_\Sigma f= \mathrm{div}_\Sigma \nabla_\Sigma f= -\mathrm{div}_\Sigma(J_{\mathbb{S}}(\mathbf{H}_\Sigma^N)).
$$
By taking the divergence of both sides we obtain
$$
m\nabla_\Sigma \Delta_\Sigma f +m \mathrm{Ric}_\Sigma(\nabla_\Sigma f) = \nabla_\Sigma \Delta_\Sigma f.
$$
The fact that ${\mathbf{U}}_\Sigma^N$, $\mathbf{A}^{\hat{N}}_\Sigma$ vanish, and \eqref{GaussEqnUeqn} with Lemma \ref{RicciULem}, give
%\begin{align*}
%\mathrm{Ric}_{\Sigma}(X,Y)&= \left( (m-1) g_{\Sigma}(X,Y)+ \frac{m+1}{(m+2)^2}|\mathbf{H}_\Sigma|^2 g_{\Sigma}(X,Y)-\frac{|\mathbf{H}^N_{\Sigma}|^2}{(m+2)^2}g_\Sigma(X,Y) \\ &+\frac{m-2}{(m+2)^2}g_\Sigma(J_{\Real}(X), \mathbf{H}_\Sigma)g_\Sigma(J_{\Real}(Y), \mathbf{H}_\Sigma).
%\end{align*}
%As we have already ensured $\mathbf{H}^{\hat{N}}_\Sigma=\mathbf{0}$ this simplifies to
\begin{align*}
\mathrm{Ric}_{\Sigma}(X,Y)&= \left( m-1 + \frac{m}{(m+2)^2}|\mathbf{H}_\Sigma^N|^2\right) g_{\Sigma}(X,Y) +\frac{m-2}{(m+2)^2}\beta_\Sigma(X)\beta_\Sigma(Y).
\end{align*}
As $m>1$, it follows that
%\begin{align*}
%	(1-m)\nabla_\Sigma \Delta_\Sigma f&=m(m-1+\frac{2m-2}{(m+2)^2} |\nabla_\Sigma f|^2) \nabla_\Sigma f.
%\end{align*}
%As $m>1$, this reduces to
\begin{equation}\label{GradEqn}
-\nabla_\Sigma \Delta_\Sigma f=m(1+\frac{2}{(m+2)^2} |\nabla_\Sigma f|^2) \nabla_\Sigma f.
\end{equation}

Finally, by Proposition \ref{NormalizeMCProp}, one may assume $d f$ and $\Delta_\Sigma f$ both vanish at $p$.  It follows from \eqref{ConfHessEqn} that $df$ vanishes to first order at $p$.  In fact, differentiating \eqref{ConfHessEqn} and using \eqref{GradEqn} inductively shows that $df$ vanishes to infinite order at $p$.
Taking the divergence of \eqref{GradEqn}, setting $\mu=\Delta_\Sigma f$, and using \eqref{ConfHessEqn} yields
$$
-\Delta_\Sigma \mu = m(1+\frac{2}{(m+2)^2} |\nabla_\Sigma f|^2)\mu + \frac{4}{(m+2)^2} |\nabla_\Sigma f|^2 \mu.
$$ 
This equation is of the form
$
|\Delta_\Sigma \mu|\leq C |\mu|
$
and so has a unique continuation property -- e.g., \cite{Aronszajn}.  Hence, as $\Sigma$ is connected and $\mu$ vanishes to infinite order at $p$, $\mu$ vanishes identically.  It follows from \eqref{ConfHessEqn} that $df$ also vanishes identically resulting in the vanishing of  $\mathbf{H}_\Sigma$ from which one concludes $\Sigma$ is totally geodesic.
\end{proof}

\section{Lower bounds for CR-volume and rigidity}
\subsection{Alternate form of CR-volume}\label{CRVolSec}
Suppose that $\Sigma \subset \mathbb{S}^{2n+1}$ is an $m$-dimensional horizontal submanifold.
Recall,  elements ${\Psi}\in Aut_{CR}(\mathbb{S}^{2n+1})$ can be factored as
$
\Psi=\Psi_{A}\circ\Psi_\mathbf{b}
$
where $A\in \mathbf{U}(n+1)$, $\mathbf{b}\in  \mathbb{B}_{\mathbb{C}}^{n+1}$.  As ${\Psi}_{A}$ is an isometry of $g_{\mathbb{S}}$, \eqref{ConformalMetEqn}, and the fact that $\Sigma$ is horizontal imply that \eqref{CRVolEqn} can be expressed as:
\begin{align*}
\lambda_{CR}[\Sigma]&=\sup_{\mathbf{b}\in \mathbb{B}_{\mathbb{C}}^{n+1}} |\Psi_{\mathbf{b}}(\Sigma)|_{\mathbb{S}}=\sup_{\mathbf{b}\in \mathbb{B}_{\mathbb{C}}^{n+1}} \int_{\Sigma} W_{\mathbf{b}}^{\frac{m}{2}}(p) dV_{\Sigma}(p)\\
&=\sup_{\mathbf{b}\in \mathbb{B}_{\mathbb{C}}^{n+1}} \int_{\Sigma} \frac{(1-|\mathbf{b}|^2)^{\frac{m}{2}}}{|1+\bar{\mathbf{b}}\cdot \mathbf{z}(p)|^{m}} dV_{\Sigma}(p).
\end{align*}
This expression is useful for some computations.
\begin{comment}
\subsection{$CR$-Volume}

We first present the definitions of horizontal and Legendrian submanifolds and $CR$-volume. 

\begin{defn}\label{def_HorLeg}
Let $\Gamma\subset \mathbb{S}^{2n+1}$ be a $m$-dimensional closed submanifold. We say that it is \emph{horizontal} if $T_p \Sigma \subset \mathcal{H}_p$ for all $p\in \Sigma$, where $\mathcal{H}_p$ is as defined in (\ref{Hp}). The properties of CR-manifolds ensure that $m\leq n$ and when $m=n$ we say that $\Gamma$ is \emph{Legendrian}. We define the \emph{CR-volume} of $\Gamma$ as
$$
\lambda_{CR}[\Gamma]=\sup_{{\Psi}\in Aut_{CR}(\mathbb{S}^{2n+1})}|{\Psi}(\Gamma)|_{\mathbb{S}}.
$$
\end{defn}

Observe that any element ${\Psi}\in Aut_{CR}(\mathbb{S}^{2n+1})$ can be factored as
$$
\Psi=\Psi_{A}\circ\Psi_\mathbf{b}
$$
where $A\in \mathbf{U}(n+1)$ and $\mathbf{b}\in  \mathbb{B}_{\mathbb{C}}^{n+1}$. Using this fact, we obtain
\begin{align*}
\lambda_{CR}[\Gamma]&=\sup_{\mathbf{b}\in \mathbb{B}_{\mathbb{C}}^{n+1}} |\Psi_{\mathbf{b}}(\Gamma)|_{\mathbb{S}}=\sup_{\mathbf{b}\in \mathbb{B}_{\mathbb{C}}^{n+1}} \int_{\Gamma} W_{\mathbf{b}}^{k/2}(p) d_{\Gamma}(p)\\
&=\sup_{\mathbf{b}\in \mathbb{B}_{\mathbb{C}}^{n+1}} \int_{\Gamma} \frac{(1-|\mathbf{b}|^2)^{k/2}}{|1+\bar{\mathbf{b}}\cdot \mathbf{z}(p)|^{k}} dV_{\Gamma}(p).
\end{align*}
Note that ${\Psi}_{A}$ is an isometry of $g_{\mathbb{S}}$ and so does not change the volume.

\end{comment}

\subsection{Rigidity of $\lambda_{CR}$-minimizers}
Suppose that $\Sigma\subset \mathbb{S}^{2n+1}$ is horizontal. We prove Theorem \ref{intro_thm_hori} by adapting an argument sketched by Bryant \cite{bryantSurfacesConformalGeometry1988}.   That is, by obtaining an expansion of the area of $\Psi_{\mathbf{b}}(\Sigma)$ as $\mathbf{b}$ approaches a point of $\Sigma$.
%\begin{prop}\label{CRAreaRigidityProp}
%	Let $\Sigma\subset \mathbb{S}^{2n+1}$ be a closed horizontal submanifold of dimension $1\leq m \leq n$.  One has
%	$$
%	\lambda_{CR}[\Sigma]\geq |\mathbb{S}^m|_{\mathbb{R}}
%	$$
%	and the equality is strict unless $\Sigma$ is $CR$-umbilic.  Moreover, there is always a $\mathbf{b}\in \mathbb{B}_{\mathbb{C}}^{n+1}$ such that
%	$$
%	\lambda_{CR}[\Sigma]=|{\Psi}_{\mathbf{b}}(\Sigma)|_{\mathbb{S}}.
%	$$
%\end{prop}

First, we record some auxiliary results.
\begin{lem}\label{SexticLem}
	Let $C^{ijk}$ be totally symmetric in $1\leq i,j,k\leq m$.  
	We have	
	$$
	\int_{\mathbb{S}^{m-1}}\left( \sum_{i,j,k=1}^m C^{ijk}\xi_i \xi_j \xi_k\right)^2 d\xi = \frac{9 |\mathbb{S}^{m-1}|_\Real}{m(m+2)(m+4)} \left( \frac{2}{3} \Vert C\Vert^2+|\mathrm{tr}(C)|^2\right)$$
	where
	$$\mathrm{tr}(C)^i=\sum_{j=1}^m C^{ijj}=\sum_{j=1}^m C^{jij}=\sum_{j=1}^m C^{jji} \mbox{ and } 	\Vert C\Vert^2=  \sum_{i,j,k=1}^m C^{ijk} C^{ijk}.
	$$
\end{lem}

\begin{proof}
	We recall certain integral identities of homogenous polynomials -- see \cite{FollandPolynomial}.
	For $i, j$ and $k$ distinct,
	$$
	\int_{\mathbb{S}^{m-1}} \xi_i^2 \xi_j^2 \xi_k^2d \xi= \frac{2\Gamma(\frac{3}{2})^3\Gamma(\frac{1}{2})^{m-3}}{\Gamma(3+\frac{m}{2})}=\frac{1}{m(m+2)(m+4)} |\mathbb{S}^{m-1}|_\Real.
	$$
	For $i$ and $j$ distinct,
	$$
\int_{\mathbb{S}^{m-1}} \xi_i^4 \xi_j^2 d \xi= \frac{2\Gamma(\frac{5}{2})\Gamma(\frac{3}{2})\Gamma(\frac{1}{2})^{m-2}}{\Gamma(3+\frac{m}{2})}=\frac{3}{m(m+2)(m+4)} |\mathbb{S}^{m-1}|_\Real.
	$$
Finally, 
	$$
\int_{\mathbb{S}^{m-1}} \xi_i^6 d \xi= \frac{2\Gamma(\frac{7}{2})\Gamma(\frac{1}{2})^{m-1}}{\Gamma(3+\frac{m}{2})}=\frac{15}{m(m+2)(m+4)} |\mathbb{S}^{m-1}|_\Real,
$$
and for any other homogeneous degree 6 polynomial, $H$,
$$
\int_{\mathbb{S}^{m-1}} H(\xi) d\xi=0.
$$
The result follows by appropriate bookkeeping.
\end{proof}

Using these results we obtain an asymptotic expression for degenerating CR-deformations of a horizontal submanifold.
\begin{prop}\label{AsympProp}
	Let $\Sigma\subset \mathbb{S}^{2n+1}$ be an $m$-dimensional horizontal submanifold that is properly embedded in a neighborhood of $p\in \Sigma$.  For $\Psi_{-(1-t)\mathbf{X}(p)}\in Aut_{CR}(\mathbb{S}^{2n+1})$,  the following asymptotic behavior holds as $t\to 0^+$, 
	$$
 |\Psi_{-(1-t)\mathbf{X}(p)}(\Sigma)|_{\mathbb{S}}=|\mathbb{S}^m|_{\Real}+ \frac{1}{4}|\mathbb{S}^m|_{\Real}  \alpha_m^\Sigma(p) \left\{\begin{array}{cc} 
t (-\log t)+O(t) & m=2\\
	\frac{2}{m-2}t +O(t^{\frac{3}{2}}) & m\geq 3\end{array}
	\right.
	$$
	where
%	$$
%	\alpha^2_\Sigma= \frac{8}{9}|\mathbf{U}_\Sigma^{N}|^2  +  |\mathring{\mathbf{A}}_\Sigma^{\hat{N}}|^2,
%	 $$
%	 and, for $m\geq 3$,
\begin{align*}
 \alpha_m^\Sigma&=	\frac{3m+2}{3m+3} |\mathbf{U}_\Sigma^N|^2+|\mathring{\mathbf{A}}_\Sigma^{\hat{N}}|^2  -\frac{m-2}{2m}\left(\frac{m^2}{(m+2)(m+1)} |\mathbf{H}_\Sigma^N|^2+ |\mathbf{H}_\Sigma^{\hat{N}}|^2\right).
\end{align*}
\end{prop}
\begin{proof}	
   Up to rotation by an element of $\mathbf{U}(n+1)$,  we may suppose $\mathbf{X}(p)=\mathbf{e}_1$ and that  $T_p\Sigma$ is spanned by $\set{\mathbf{e}_2, \ldots, \mathbf{e}_{m+1}}$.  
   Observe that for $t\in (-1,1)$,
   $$
   |\Psi_{-(1-t)\mathbf{e}_1}(\Sigma)|_{\mathbb{S}}=\int_{\Sigma} \frac{(1-(1-t)^2)^{\frac{m}{2}}}{((1-(1-t) x_1)^2+(1-t)^2 y_1^2)^{\frac{m}{2}}} dV_\Sigma.
   $$
   Moreover, the hypothesis that $\Sigma$ is properly embedded near $p$ means for small enough $\epsilon>0$,  we may  pick coordinates $$\mathbf{w}=(w_1, \ldots, w_m)\in B_{2\epsilon}(\mathbf{0})\subset \Real^m$$
    and functions $u_1, u_2$ and ${z}_1 \ldots, z_{2n-m}$ so that 
    $$
    \set{(1+u_1(\mathbf{w}), \mathbf{w}, z_1(\mathbf{w}), \ldots,z_{n-m}(\mathbf{w}), u_2(\mathbf{w}), z_{n+1-m}, \ldots } \subset \Sigma
    $$ 
	where at $\mathbf{0}\in \Real^m$, $u_1=u_2=z_1=\ldots= z_{2n-m}=0$ and $\nabla u_1=\nabla u_2=\nabla z_1=\ldots =\nabla z_{2n-m}=0$.
	Using 
	$$
	\nabla^2_\Sigma x_1 (p) = \mathbf{e}_1\cdot \mathbf{A}_{\Sigma}^{\Real}(p)=-g_{\Sigma} (p)
	$$
	we see that
	$$
	u_1(\mathbf{w}) =-\frac{1}{2}|\mathbf{w}|^2+O(|\mathbf{w}|^3).
	$$
	Likewise, as $\Sigma$ is horizontal, 
	$$
	\nabla^2_\Sigma y_1(p)=\mathbf{T}\cdot \mathbf{A}_{\Sigma}^{\Real} (p)= 0.
	$$
 This means 
	$$
	u_2(\mathbf{w})=C^{ijk} w_i w_j w_k+O(|\mathbf{w}|^4) \mbox{ where }
	C^{ijk}=\frac{1}{6}\sigma_\Sigma|_{p}(\mathbf{e}_{i+1}, \mathbf{e}_{j+1}, \mathbf{e}_{k+1}).
	$$
	Hence, for $|\mathbf{w}|$ small, one has the estimate
	$
	|u_2|^2 \leq C |u_1|^3
	$
	for some constant $C$.  
	 That is,  for $q\in \Sigma$ sufficiently close to $p$, one has 
	$$
	|y_1(q)|^2\leq C (1-x_1(q))^3.
	$$
	 As $t>0$ and $x_1(q)\geq 0$ for $q$ near enough to $p$, one has
	$$
	0\leq 1-x_1(q)<1-(1-t) x_1(q), \mbox{ so }	y_1^2(q)\leq C(1-(1-t) x_1(q))^3.
	$$

	Fix an $\epsilon\in (0,1)$. The preceding estimates ensure that for $t>0$ small,
	\begin{align*}
	\left|	|\Psi_{-(1-t)\mathbf{e}_1}(\Sigma)|_{\mathbb{S}}-	F_\epsilon^m(t)-G_{\epsilon}^m(t)\right|\leq C_m t F_{\epsilon}^{m-2}(t)+ O(t^{\frac{m}{2}}; \epsilon)
	\end{align*}
	where
	$$
	F_\epsilon^l(t)=\int_{\Sigma\cap \set{x_1\geq \epsilon}} \frac{(1-(1-t)^2)^{\frac{l}{2}}}{(1-(1-t)x_1)^l} dV_{\Sigma} \mbox{ and }
	$$
$$
G_\epsilon^l(t)=-\frac{l}{2}(1-t)^2 \int_{\Sigma\cap \set{x_1\geq \epsilon}} \frac{ y_1^2 (1-(1-t)^2)^{\frac{l}{2}}}{(1-(1-t)x_1)^{l+2}} dV_{\Sigma}.
$$
Note that we also have
$$
|G_\epsilon^l(t)|\leq C_{l, \Sigma} t^{\frac{1}{2}} F_{\epsilon}^{l-1}(t).
$$	

As $\Sigma\subset \mathbb{S}^{2n+1}$ and $\mathbf{X}(p)=\mathbf{e}_1$, 
$$
\Sigma(s)=\Sigma\cap	\set{x_1\geq s}=  \bar{B}_{\sqrt{2(1-s)}}^\Real(p)\cap \Sigma.
$$
Using the co-area formula twice and integrating by parts we obtain
	\begin{align*}
		F_{\epsilon}^l(t) &= t^{\frac{l}{2}} (2-t)^{\frac{l}{2}} \int_{\epsilon}^1  \int_{\Sigma\cap \set{x_1=s}} \frac{1}{(1-(1-t) s)^l}  \frac{1}{|\nabla_\Sigma x_1|} dV_s ds \\
		&=t^{\frac{l}{2}} (2-t)^{\frac{l}{2}} \int_{\epsilon}^1  \frac{1}{(1-(1-t) s)^l} \left(-\frac{d}{ds} |\Sigma(s)|_{\Real}\right) ds\\
		&=	t^{\frac{l}{2}} (2-t)^{\frac{l}{2}}  \left(\frac{ |\Sigma (\epsilon)|_\Real}{(1-(1-t) \epsilon)^l} + l (1-t)\int_\epsilon^1 \frac{|\Sigma(s)|_\Real}{{(1-(1-t) s)^{l+1}}}ds\right).
	\end{align*}
From this we conclude that $F_{\epsilon}^l(t) =\hat{F}_\epsilon^l(t)+O(t^{\frac{l}{2}}; \epsilon)$ where 
\begin{align*}
\hat{F}_\epsilon^l(t) = 	t^{\frac{l}{2}} (2-t)^{\frac{l}{2}} (1-t)l\int_\epsilon^1 \frac{|\Sigma(s)|_\Real}{{(1-(1-t) s)^{l+1}}}ds.
\end{align*}
A computation of Karp and Pinksy \cite{karpVolumeSmallExtrinsic1989}, see also \cite[Appendix A]{Bernstein2023rigidity}, yields
	\begin{align*}
	 | \Sigma(s)|_\Real= A_m (1-s)^{\frac{m}{2}}+B_{m}^\Sigma (1-s)^{\frac{m}{2}+1}+O((1-s)^{\frac{m}{2}+\frac{3}{2}}),
	\end{align*}
	where denoting $\omega_m$ as the volume of the unit ball in $\Real^m$, 
\begin{align*}
	A_m 	&= 2^{\frac{m}{2}}  \frac{1}{m}|\mathbb{S}^{m-1}|_\Real=2^{\frac{m}{2}}  \omega_m, \\
	B_{m }^\Sigma&= 2^{\frac{m+2}{2}}\frac{ 1}{8m(m+2)}  |\mathbb{S}^{m-1}|_\Real \left(2|\mathbf{A}_\Sigma^\Real(p)|^2-|\mathbf{H}_\Sigma^\Real(p)|^2\right)\\
	&=  \frac{1 }{4(m+2)}A_m  \left(2|\mathbf{A}_\Sigma^{\mathbb{S}}(p)|^2-|\mathbf{H}_\Sigma^{\mathbb{S}}(p)|^2+2m-m^2\right).
\end{align*}
Thus, $\hat{F}_\epsilon^l(t) $ has an expansion using the functions from Lemma \ref{horizontalWithtLem} of the form,
\begin{align*}
	 t^{\frac{l}{2}}\big(2-t \big)^{\frac{l}{2}} (1-t) l \left(  A_m I_{l+1,m+2}(t; \epsilon)+ B_{m}^\Sigma I_{l+1, m+4} (t; \epsilon)+O(I_{l+1, m+5}(t; \epsilon)\right).
\end{align*}
For $2l\geq m+1$, Lemma \ref{horizontalWithtLem} implies
the leading order behavior
$$
	\hat{F}_\epsilon^l(t)=2^{\frac{l+m}{2}}\frac{l}{m} |\mathbb{S}^{m-1}|_{\Real} 
C_{l+1, m+2} t^{-\frac{l}{2}+\frac{m}{2}}+O(t^{1-\frac{l}{2}+\frac{m}{2}}).
$$
Lemma \ref{horizontalWithtLem}  also yields the following required special cases not covered by this
\begin{align*}
\hat{F}_\epsilon^l(t)=\left\{ \begin{array}{cc}  O(1; \epsilon)  &\mbox{$m=2$ and $0\leq l \leq 1$}\\
O(t^{\frac{1}{2}}; \epsilon) &\mbox{$m=3$ and $l=1$}. \end{array} \right.
\end{align*}
It follows that, for all $m\geq 2$, 	$|\Psi_{-(1-t)\mathbf{e}_1}(\Sigma)|_{\mathbb{S}}=\hat{F}_\epsilon^m(t)+O(t^{\frac{1}{2}}; \epsilon).$
When $\Sigma$ is a totally geodesic horizontal sphere $	|\Psi_{-(1-t)\mathbf{e}_1}(\Sigma)|_{\mathbb{S}}=|\mathbb{S}^m|_\Real$, for $l=m\geq 2$, 
\begin{equation}\label{Cm1m2Eqn}
|\mathbb{S}^m|_\Real=\hat{F}_\epsilon^m(0)= 2^m|\mathbb{S}^{m-1}|_{\Real} C_{m+1,m+2}, \mbox{ i.e., }C_{m+1,m+2}=2^{-m} \frac{|\mathbb{S}^m|_\Real}{|\mathbb{S}^{m-1}|_\Real}
\end{equation}
and so, for general $\Sigma$, $\hat{F}_\epsilon^m(0)=|\mathbb{S}^m|_\Real$.  \
 Another consequence is that when $m\geq 2$, 
$$
|\Psi_{-(1-t)\mathbf{e}_1}(\Sigma)|_{\mathbb{S}}=|\mathbb{S}^m|_\Real+\hat{\alpha}_m^\Sigma t +\left\{ \begin{array}{cc} O(t) & m=2\\ O(t^{\frac{3}{2}}) & m \geq 3 \end{array}\right.
$$
where $\hat{\alpha}_m^\Sigma$ is computed from the leading order term of $G^m_\epsilon (t)$ and the subleading order term of $\hat{F}_\epsilon(t)$.

Let us now suppose $l=m\geq 3$. 
In this case, we have
\begin{align*}
 &t^{\frac{m}{2}}\big(2-t \big)^{\frac{m}{2}} (1-t)m= 2^{\frac{m}{2}}m (1-\frac{m+4}{4} t)+O(t^2)\\
 &  A_m I_{m+1,m+2}(t; \epsilon)= A_{m} C_{m+1, m+2} t^{-\frac{m}{2}}+ A_m \frac{m+2}{2} C_{m+1, m+2}  t^{1-l+\frac{m}{2}}+O( t^{2-\frac{m}{2}})\\
 &B_{m}^\Sigma I_{m+1, m+4}(t; \epsilon)=B_{m}^\Sigma C_{m+1, m+4} t^{1-l+\frac{m}{2}}+O( t^{2-\frac{m}{2}}).
\end{align*}
Hence, when $m\geq 3$, the next term in the expansion of $\hat{F}^{m}_\epsilon(t)$ has coefficient
\begin{align*}
 2^{\frac{m}{2}}& m\big( 	\frac{m+2}{2}A_m C_{m+1, m+2} 	+ B_{m}^\Sigma C_{m+1,m+4}-\frac{m+4}{4} A_m  C_{m+1, m+2}\big)\\
 &=  |\mathbb{S}^m|_\Real \big(\frac{m}{4}+ \frac{1}{4(m+2)}(2|\mathbf{A}_\Sigma^{\mathbb{S}}|^2-|\mathbf{H}_\Sigma^{\mathbb{S}}|^2+2m-m^2)\frac{C_{m+1, m+4}}{C_{m+1, m+2}}\big) 	.
 	\end{align*}
When $\Sigma$ is the totally geodesic horizontal sphere, $y_1$ is identically zero, so $G^m_{\epsilon}(t)=0$, and hence this term vanishes.  Moreover,  $2|\mathbf{A}_\Sigma^{\mathbb{S}}|^2-|\mathbf{H}_\Sigma^{\mathbb{S}}|^2=0$  and so
\begin{align}\label{cm1m4eqn}
\frac{C_{m+1, m+4}}{C_{m+1, m+2}}= \frac{ m+2}{m-2}.
\end{align}
Hence, when $m\geq 3$
the expansion to subleading order is given by
%\begin{align*}
% 	 \frac{|\mathbb{S}^{m}|_\Real }{4(m-2)}  \left(2 |\mathbf{A}_{\Sigma}^{\mathbb{S}}|^2-|\mathbf{H}_\Sigma^{\mathbb{S}}|^2\right).
%\end{align*}
%That is,
$$
\hat{F}_\epsilon^m(t)=|\mathbb{S}^m|_\Real +\frac{|\mathbb{S}^{m}|_\Real }{4(m-2)}  \left(2 |\mathbf{A}_{\Sigma}^{\mathbb{S}}|^2-|\mathbf{H}_\Sigma^{\mathbb{S}}|^2\right)t +O(t^{\frac{3}{2}}).
$$
%We also have
%$$
%C_{m+1, m+4}=\frac{ m+2}{m-2} C_{m+1, m+2}=2^{-m} \frac{ m+2}{m-2} \frac{|\mathbb{S}^m|_\Real}{|\mathbb{S}^{m-1}|_\Real}
%$$

Next, we treat the term $G_\epsilon^m$ for $m\geq 2$.  To that end, we observe that the co-area formula implies
\begin{align*}
G_\epsilon^m(t) &= -\frac{m}{2}(1-t)^2  t^{\frac{m}{2}} (2-t)^{\frac{m}{2}}\int_{\epsilon}^1  \int_{\Sigma\cap \set{x_1=s}} \frac{y_1^2}{(1-(1-t) s)^{m+2}}  \frac{1}{|\nabla_\Sigma x_1|} dV_s ds \\
&=-\frac{m}{2}(1-t)^2  t^{\frac{m}{2}} (2-t)^{\frac{m}{2}}\int_{\epsilon}^1   \frac{H(s)}{(1-(1-t) s)^{m+2}}  ds 
\end{align*}
where
$$
H(s)= \int_{\Sigma\cap \set{x_1=s}}  \frac{y_1^2}{|\nabla_\Sigma x_1|} dV_s.
$$
It follows from Lemma \ref{SexticLem} that
$$
H(s)= (2(1-s))^{\frac{m+4}{2}} \frac{|\mathbb{S}^{m-1}|_\Real}{4m(m+2)(m+4)}\big(\frac{2}{3}|\mathbf{A}_\Sigma^N|^2+|\mathbf{H}_\Sigma^N|^2\big)+O((1-s)^{\frac{m+5}{2}})
$$
and so by Lemma \ref{horizontalWithtLem} and the fact that for $m\geq 2$, $I_{m+2, m+7}(t; \epsilon)=O(t^{\frac{3}{2}-\frac{m}{2}})$,
$$
	G_\epsilon^m(t) = -2^{m-1} t^{\frac{m}{2}} \frac{|\mathbb{S}^{m-1}|_{\Real}}{(m+2)(m+4)}\big(\frac{2}{3}|\mathbf{A}_\Sigma^N|^2+|\mathbf{H}_\Sigma^N|^2\big) I_{m+2, m+6}(t; \epsilon)+O(t^{\frac{3}{2}}).
$$			
 
Hence, by Lemma \ref{horizontalWithtLem}, when $m\geq 3$, one has
\begin{align*}
	G_\epsilon^m(t) &= -2^{m-1}\frac{|\mathbb{S}^{m-1}|_\Real}{(m+2)(m+4)}\big(\frac{2}{3}|\mathbf{A}_\Sigma^N|^2+|\mathbf{H}_\Sigma^N|^2\big) C_{m+2, m+6} t+O(t^{\frac{3}{2}}).
\end{align*}
When $m\geq 3$, Lemma \ref{horizontalWithtLem}, \eqref{Cm1m2Eqn}, and \eqref{cm1m4eqn} yield: 
\begin{align*}
C_{m+2,m+6}&=C_{m+1, m+4}-C_{m+2, m+4}=\frac{m+2}{m-2} C_{m+1, m+2}-\frac{m+2}{2(m+1)}C_{m+1, m+2}\\
&=\frac{ (m+2)(m+4)}{2(m-2)(m+1)} C_{m+1, m+2}= 2^{-m-1} \frac{ (m+2)(m+4)}{(m-2)(m+1)} \frac{|\mathbb{S}^m|_\Real}{|\mathbb{S}^{m-1}|_{\Real}}.
\end{align*}
Hence, for $m\geq 3$, 
\begin{align*}
	G_\epsilon^m(t) &=- \frac{|\mathbb{S}^{m}|_\Real}{4(m-2)(m+1)}\big(\frac{2}{3}|\mathbf{A}_\Sigma^N|^2+|\mathbf{H}_\Sigma^N|^2\big) t+O(t^{\frac{3}{2}}),
\end{align*}
and so combining the leading order and subleading order terms gives
\begin{align*}
&\frac{2(m-2)}{|\mathbb{S}^m|_\Real}\hat{\alpha}_m^\Sigma=2\left(-\frac{1}{3(m+1)} \left( 2|\mathbf{A}_\Sigma^N|^2+ 3|\mathbf{H}_\Sigma^N|^2\right)+ 2|{\mathbf{A}}^{\mathbb{S}}_\Sigma|^2-|\mathbf{H}_\Sigma^{\mathbb{S}}|^2\right)\\
&=\left(\frac{3m+2}{3m+3} |\mathbf{U}_\Sigma^N|^2+|\mathring{\mathbf{A}}_\Sigma^{\hat{N}}|^2-\frac{ m(m-2)}{2(m+1)(m+2)} |\mathbf{H}_\Sigma^N|^2-\frac{m-2}{2m} |\mathbf{H}_\Sigma^{\hat{N}}|^2\right)
\end{align*}
from which we deduce the claimed $\alpha_m^\Sigma$. Here we used 
\begin{equation} \label{UAHeqn}
|\mathbf{U}_\Sigma^N|^2=\mathrm{tr}_\Sigma(\mathbf{U}_\Sigma^N\odot \mathbf{U}_\Sigma^N) =|\mathbf{A}_\Sigma^N|^2-\frac{3}{m+2} |\mathbf{H}_\Sigma^N|^2.
\end{equation}

Finally, let us treat the expansion in the $m=2$ case. By Lemma \ref{horizontalWithtLem}, 
$$
I_{3,6}(t;\epsilon)=I_{4,8}(t; \epsilon)=-\log t+O(1; \epsilon).
$$
It follows that
\begin{align*}
	\hat{F}_\epsilon^2(t)&=-4 t \log t B_m^\Sigma+O(t; \epsilon)\\
	&=|\mathbb{S}^2|_\Real+\frac{|\mathbb{S}^2|_\Real}{8}  \left(2|\mathbf{A}_\Sigma^{\mathbb{S}}|^2-|\mathbf{H}_\Sigma^{\mathbb{S}}|^2\right) t(-\log t)+O(t; \epsilon);
\end{align*}
\begin{align*}
	G_\epsilon^2(t)&=\frac{|\mathbb{S}^1|_{\Real}}{12} \big(\frac{2}{3}|\mathbf{A}_\Sigma^N|^2+|\mathbf{H}_\Sigma^N|^2\big)t \log t+O(t^{\frac{3}{2}}; \epsilon)\\
	&=-\frac{|\mathbb{S}^2|_{\Real}}{72} (2|\mathbf{A}_\Sigma^N|^2+3|\mathbf{H}_\Sigma^N|^2)t (-\log t)+O(t^{\frac{3}{2}}; \epsilon).
\end{align*}
Adding these and using \eqref{UAHeqn} yields
$$
|\Psi_{-(1-t)\mathbf{e}_1}(\Sigma)|_{\mathbb{S}}=|\mathbb{S}^2|_{\Real}+\frac{1}{4}|\mathbb{S}^2|_{\Real}\left( \frac{8}{9} |\mathbf{U}_\Sigma^N|^2 +  |\mathring{\mathbf{A}}_\Sigma^{\hat{N}}|^2 \right) t (-\log t)+O(t).
$$
\end{proof}

\begin{lem}\label{AntipodLem}
	Let $p_\pm\in \mathbb{S}^{2n+1}$ be two distinct points.  There is an element $\Psi \in Aut_{CR}(\mathbb{S}^{2n+1})$ such that $\Psi(p_-)$ and $\Psi(p_+)$ are antipodal.
\end{lem}
\begin{proof}
	Consider $\mathbb{B}^{2n+2}$ together with the Bergman metric $g_B$ -- see \cite{GoldmanCHBook} or \cite{BernBhattCplxHyp}.    It is a standard fact that there is a complete unit speed geodesic of $g_B$,  $\gamma:(-\infty, \infty)\to \mathbb{B}^{2n+2}$ such that $\lim_{t\to \pm \infty}\gamma(t) =p_\pm$. The group of oriented isometries of $g_B$ acts transitively on $\mathbb{B}^{2n+2}$ and so there is an isometry $\Phi$ such that $\Phi(\gamma(0))=\mathbf{0}$.  In particular, $\Phi\circ \gamma$ is a parameterization of the line segment through $\mathbf{0}$.   If $\Psi$ is the corresponding element in $Aut_{CR}(\mathbb{S}^{2n+1})$ it follows that $\Psi(p_\pm)$ are the limits of this segment and hence are antipodal.
\end{proof}

We are now able to prove Theorem \ref{intro_thm_hori}.
\begin{proof}[Proof of Theorem \ref{intro_thm_hori}]
		Pick a $p\in \Sigma$,  by  Proposition \ref{AsympProp} 
	$$
	\lambda_{CR}[\Sigma]\geq \lim_{t\to 0^+} |\Psi_{-(1-t) \mathbf{X}(p)}(\Sigma)|_{\mathbb{S}}=|\mathbb{S}^m|_{\mathbb{R}}.
	$$	
	By Lemma \ref{AntipodLem}, there is a $\Psi\in Aut_{CR}(\mathbb{S}^{2n+1})$ such that $\Psi(p)=p$ and $-p \in \Psi(\Sigma)$.  That is, $\Psi(\Sigma)$ contains a pair of antipodal points. 
	When $m=1$,  as $\Sigma$ is closed, 
	$$
	\lambda_{CR}[\Sigma]\geq |\Psi(\Sigma)|_{\mathbb{S}}\geq 2 \dist_{\mathbb{S}}(p,-p)=2\pi=|\mathbb{S}^1|_{\Real}
	$$
	with equality in the second inequality if and only if $\Psi(\Sigma)$ is a single closed geodesic.
	
	When $m\geq 2$,   Proposition \ref{AsympProp} implies that
	$
	\lambda_{CR}[\Sigma]=|\mathbb{S}^m|_{\Real}
	$
	can only occur when 
	$$
	\frac{3m+2}{3m+3}|\mathbf{U}_\Sigma^N|^2+|\mathring{\mathbf{A}}_\Sigma^{\hat{N}}|^2\equiv {0},
	$$
	which occurs only when $\mathbf{U}_\Sigma^N$ and $\mathring{\mathbf{A}}_\Sigma^{\hat{N}}$ both vanish.
	Indeed, for any $p\in \Sigma$ we may, by Proposition \ref{NormalizeMCProp}, choose a $\Psi\in  Aut_{CR}(\mathbb{S}^{2n+1})$ such that $\Psi(p)=p$ and $\mathbf{H}_{\Psi(\Sigma)}(p)=\mathbf{0}$.  For equality to hold, the quantity $\alpha_m^\Sigma(p)$ from  Proposition \ref{AsympProp} must vanish as otherwise, for small enough $t$, the subleading order term is positive and $\lambda_{CR}[\Psi(\Sigma)]=\lambda_{CR}[\Sigma]>|\mathbb{S}^m|_\Real$.  However, by the invariance properties and vanishing of the mean curvature of $\Psi(\Sigma)$ at $p$ one has
	$$
 \alpha^{\Psi(\Sigma)}_m(p)= 	\frac{3m+2}{3m+3}|\mathbf{U}_{\Psi(\Sigma)}^N|^2(p)+|\mathring{\mathbf{A}}_{\Psi(\Sigma)}^{\hat{N}}|^2(p)=		\frac{3m+2}{3m+3}|\mathbf{U}_\Sigma^N|^2(p)+|\mathring{\mathbf{A}}_\Sigma^{\hat{N}}|(p).
 $$
 Thus, $\mathbf{U}_{\Sigma}^N$ and $\mathring{\mathbf{A}}_\Sigma^{\hat{N}}$ identically vanish and so Theorem  \ref{CRinv} establishes the lower bound and rigidity claim -- in particular equality holds only when $\Sigma$ has one component.  It also establishes that the supremum is achieved when $\lambda_{CR}[\Sigma]=|\mathbb{S}^m|_{\Real}$. 
	
To see that the supremum is otherwise achieved we may assume  $\lambda_{CR}[\Sigma]>|\mathbb{S}^m|_{\Real}$ and so, as $\Sigma$ is embedded,  Proposition \ref{AsympProp} implies that there is a compact subset $\bar{B}\subset \mathbb{B}^{2n+2}$ such that 
	$$
	\lambda_{CR}[\Sigma]=\sup_{\mathbf{b}\in \bar{B}} |\Psi_{\mathbf{b}}(\Sigma)|_{\mathbb{S}}=\max_{\mathbf{b}\in \bar{B}} |\Psi_{\mathbf{b}}(\Sigma)|_{\mathbb{S}}
	$$
 where the second equality used the fact that $\mathbf{b}\mapsto  |\Psi_{\mathbf{b}}(\Sigma)|_{\mathbb{S}}$ is continuous and $\bar{B}$ is compact.  Hence, in either case, $\lambda_{CR}[\Sigma]=|\Psi_{\mathbf{b}}(\Sigma)|_{\mathbb{S}}$ for some $\mathbf{b}\in \bar{B}$.	
 
 Finally, if $i: M \to \mathbb{S}^{2n+1}$ is a horizontal immersion, then for a small open neighborhood $U$ of $i(p)$, one has 
 $$U\cap i(M)=\cup_{j=1}^k \Sigma_j$$
 where $\Sigma_j$ are properly embedded horizontal submanifolds of $U\cap \mathbb{S}^{2n+1}$ and each contains $p$.  Proposition \ref{AsympProp} implies that
 $$
\lim_{t\to 1^-}|\Psi_{-(1-t)\mathbf{X}(p)}(\Sigma_j)|_{\mathbb{S}}=|\mathbb{S}^m|_{\Real}.
$$
The final claim immediately follows.
\end{proof}

We observe there are horizontal submanifolds whose  CR-volume is strictly smaller than their conformal volume.

\begin{prop}\label{CRConfAreaDifferProp}
	For every $1\leq m \leq n$, there is an $m$-dimensional horizontal submanifold $\Sigma \subset \mathbb{S}^{2n+1}$ such that
	$$
	\lambda_c[\Sigma]>\lambda_{CR}[\Sigma].
	$$
\end{prop}
\begin{rem}
	It would be interesting to know if it is possible for any Legendrian (or horizontal) submanifold $\Sigma$ to satisfy $\lambda_c[\Sigma]< \lambda_{CR}[\Sigma].$   
\end{rem}
\begin{proof}
	Let $\Sigma\subset \mathbb{S}^{2n+1}$ be a totally geodesic Legendrian $n$-sphere in $\Sigma$.  In particular, $\Sigma=\mathbb{S}^{2n+1}\cap P$ for a Lagrangian subspace $P\subset \Real^{2n+2}$.  By Proposition \ref{CurvatureChangeProp}, for $\mathbf{b}\in \mathbb{B}^{2n+2}\setminus P$ and corresponding $\Psi_{\mathbf{b}}\in Aut_{CR}(\mathbb{S}^{2n+1})$, $\Sigma_{\mathbf{b}}=\Psi_{\mathbf{b}}(\Sigma)$ is not a round sphere. Hence, \cite[Proposition 1]{bryantSurfacesConformalGeometry1988} implies  $\lambda_{c}[\Sigma_{\mathbf{b}}]>\lambda_c[\Sigma]=|\mathbb{S}^n|_{\mathbb{S}}=\lambda_{CR}[\Sigma_{\mathbf{b}}]$.
\end{proof}
\subsection{CR-volume as conformal invariant}

If $\phi:M \to \mathbb{S}^{2n+1}$ is a conformal horizontal immersion, then, by \eqref{ConformalMetEqn}, for all $\Psi\in Aut_{CR}(\mathbb{S}^{2n+1})$, $\Psi \circ \phi: M \to \mathbb{S}^{2n+1}$ is also a conformal horizontal immersion.  Hence, as with the Li-Yau conformal volume, to any conformal manifold $(M,[g])$ one may associate a conformal invariant
$$
V_{CR}(2n+1,M, [g] )=\inf_\phi V_{CR}(2n+1,\phi)
$$
where $\phi$ runs over all conformal horizontal immersions $\phi: M \to \mathbb{S}^{2n+1}$.  For large $n$, such immersions exist as there is a conformal immersion $\phi: M\to \mathbb{S}^n$. Thus, $i\circ \phi: M\to \mathbb{S}^{2n+1}$ is a conformal horizontal immersion where  $i: \mathbb{S}^n\to \mathbb{S}^{2n+1}$ is the inclusion map of the standard Legendrian $\mathbb{S}^n$.  Many properties of conformal area from \cite{Li1982} and subsequent works should have analogs for CR-volume. 

Here we directly adapt an argument of \cite{Li1982} and establish a lower bound on the CR-volume of surfaces in terms of their spectral geometry.
\begin{lem}\label{BalanceLem}
	Fix a closed Riemannian manifold $(M,g)$.  If $\phi: M\to \mathbb{S}^{2n+1}$ is a smooth map such that $\phi^{-1}(p)$ is finite for each $p\in \mathbb{S}^{2n+1}$, then there exists $\mathbf{b}_*\in \mathbb{B}^{2n+2}$ so that $\Psi_{\mathbf{b}_*}\circ \phi$ is balanced, i.e., 
	$$
	\int_{M} \mathbf{X}\circ \Psi_{\mathbf{b}_*} \circ \phi\;  dA_{g}=\mathbf{0}.
	$$
\end{lem}
\begin{proof}
Consider the map $	\mathbf{B}:\mathbb{B}^{2n+2}\to \mathbb{B}^{2n+2}$ given by
	$$
\mathbf{B}( \mathbf{b})= \frac{1}{|M|_g}\int_M  \mathbf{X}\circ \Psi_{\mathbf{b}}\circ \phi \;  dA_g.
	$$
	As  $\phi^{-1}(p_0)$ is finite for all $p_0\in \mathbb{S}^{2n+1}$, \eqref{ConvergencePsibEqn} ensures that, away from a finite set of points in $M$,
	$$
	\lim_{\mathbf{b}\to \mathbf{b}_0}\mathbf{X}\circ \phi=\mathbf{b}_0.
	$$
	Moreover, the convergence is uniform on any compact subset of the complement of  $\phi^{-1}(-p_0)$.
	Hence, $\mathbf{B}$ extends to a continuous map $\bar{\mathbf{B}}: \bar{\mathbb{B}}^{2n+2}\to \bar{\mathbb{B}}^{2n+2}$ satisfying
	$$
	\bar{\mathbf{B}}(\mathbf{b}_0)=\mathbf{b}_0
	$$
	for $\mathbf{b}_0\in \partial \bar{\mathbb{B}}^{2n+2}$.  
 The existence of the desired $\mathbf{b}_*$ follows -- indeed otherwise there would be a retraction from $\bar{\mathbb{B}}^{2n+2}$ to $\partial \bar{\mathbb{B}}^{2n+2}=\mathbb{S}^{2n+1}$.  
\end{proof}

\begin{prop}\label{CRVolumeLowBoundProp}
	If $(M,g)$ is a closed Riemannian surface, then
\begin{equation} \label{CRVolEigenEqn}
	V_{CR}( 2n+1,M, [g])\geq \frac{1}{2}\lambda_1(g) |M|_g
\end{equation}
 when this is defined. Here $\lambda_1(g)$ is the first non-trivial eigenvalue of $g$.  Equality holds only if there is a horizontal isometric minimal immersion $i:M\to \mathbb{S}^{2n+1}$.  
\end{prop} 
\begin{proof}
	Let $\phi: M\to \mathbb{S}^{2n+1}$ be a horizontal conformal immersion.   By Lemma \ref{BalanceLem} and the properties of $Aut_{CR}(\mathbb{S}^{2n+1})$, we may replace $\phi$ by $\Psi_{\mathbf{b}_*}\circ \phi$ and so assume $\phi$ is also balanced.
 	As $\phi$ is conformal and balanced, the variational characterization of eigenvalues gives
	$$
	|M|_{\phi^* g_{\mathbb{S}}}= \frac{1}{2}\int_M |\nabla_g \phi|^2 dA_g\geq\frac{1 }{2} \lambda_1(g)\int_M| \phi|^2 dA_g=\frac{1 }{2}  \lambda_1(g)|M|_g.
	$$
 	To conclude the proof of \eqref{CRVolEigenEqn}, observe that, by definition, for every $\epsilon>0$ there is a balanced and conformal horizontal map $\phi_\epsilon: M\to \mathbb{S}^{2n+1}$ so that
	$$
	V_{CR}(2n+1,  M, [g])+\epsilon\geq 	V_{CR}(2n+1, \phi_\epsilon)\geq 	|M|_{\phi^*_\epsilon g_{\mathbb{S}}}\geq \frac{1}{2} \lambda_1(g) |M|_g.
	$$
	As $\epsilon>0$ was arbitrary, the desired inequality follows.  
	
	If equality holds, then there is a sequence of conformal horizontal immersions $\phi_i: M\to \mathbb{S}^{2n+1}$ so that
	$$
	\frac{1}{2} \lambda_1(g) |M|_g+\frac{1}{i}\geq V_{CR}(2n+1,\phi_i) \geq \frac{1}{2} \lambda_1(g) |M|_g.
	$$
	As above, we may assume that the $\phi_i$ are balanced.  It follows that
	$$
	\lambda_1(g) |M|_g+\frac{2}{i}\geq \int_M |\nabla_g \phi_i|^2 dA_g  \geq  \lambda_1(g) |M|_g.
	$$
	Here the first inequality used the definition of CR-volume and the second used the variational property of the first eigenvalue as above.
	
	It follows that the $\phi_i$ are uniformly bounded in $W^{1,2}(M)$ and so, up to passing to a subsequence, there is a limit $\phi_\infty\in W^{1,2}(M; \mathbb{S}^{2n+1})$.  As the convergence is strong in $L^2(M; \mathbb{S}^{2n+1})$ the limit is balanced and so 
	$$ \int_M |\nabla_g \phi_\infty|^2 dA_g  \geq  \lambda_1(g) |M|_g.$$
	Hence, there is no energy loss in the limit and so the convergence is strong in $W^{1,2}$.  Moreover,  the variational characterization of $\lambda_1(g)$ implies that each component of  $\phi_\infty$ is either zero a.e. or is an eigenfunction with eigenvalue $\lambda_1$.  In either case, each component differs from a smooth function on a set of measure zero.  That is, we may take $\phi_\infty$ to be smooth.  
As $1=|\phi_\infty|^2$ and $\Delta_g \phi_\infty=-\lambda_1(g) \phi_\infty$, we have
	$$
	 0=\Delta_g |\phi_\infty|^2=2\phi_\infty \cdot \Delta_g \phi_\infty+ 2 |\nabla_g \phi_\infty|^2=-2\lambda_1(g) +2 |\nabla_g \phi_\infty|^2.
	 $$
	 That is, $|\nabla_g \phi_\infty|^2=\lambda_1(g)$ holds pointwise.
	
	The fact that the $\phi_i$ are conformal horizontal immersions,  the strong convergence in $W^{1,2}$, and the fact that $\phi_\infty$ is smooth means $\phi_\infty$ is horizontal and is weakly conformal. Moreover, as $|\nabla_g \phi_\infty|^2=\lambda_1(g)>0$,  $\phi_\infty$ is also an immersion. Finally, as $\phi_\infty$ is a map by first eigenfunctions,  it is a horizontal conformal minimal immersion and $\phi_\infty^* g_{\mathbb{S}}$ is a constant multiple of $g$. 
\end{proof}

\section{CR-Willmore Energy}
\label{CRWillmoreSec}
%While we do not use it in this paper we observe that there is a natural analog of the Willmore energy.  This concept was also studied for Legendrian surfaces in $\mathbb{S}^5$ in \cite{wangWillmoreLegendrianSurfaces2007}.

For an $m$-dimensional horizontal submanifold $\Sigma\subset \mathbb{S}^{2n+1}$, $m\geq 2$  consider
$$
\mathcal{U}_{CR}[\Sigma]= \frac{1}{m}\int_\Sigma |{\mathbf{U}}_\Sigma^N|^{m} dV_\Sigma \mbox{ and } \mathcal{B}_{CR}[\Sigma]=\frac{1}{m}\int_\Sigma |\mathring{\mathbf{A}}_\Sigma^{\hat{N}}|^{m} dV_\Sigma.
$$
Both functionals are $Aut_{CR}(\mathbb{S}^{2n+1})$ invariant.  Indeed, by Proposition \ref{UInvProp} and the fact that $\Psi\in Aut_{CR}(\mathbb{S}^{2n+1})$ acts on $\mathcal{H}$, and hence on any horizontal submanifold, as a conformal immersion, 
$$
|\mathbf{U}_{\Psi(\Sigma)}|_{g_{\Psi(\Sigma)}}^{\frac{m}{2}}(\Psi(p)) =|\mathbf{U}_{\Psi(\Sigma)}|_{\Psi^*g_{\Sigma}}^{\frac{m}{2}}(\Psi(p))=W_{\mathbf{b}}^{m} (p) |\mathbf{U}_{\Sigma}|_{g_{\Sigma}}^{\frac{m}{2}} (p)
$$
for any $\mathbf{b}\in \mathbb{B}^{2n+2}$.  The same holds for $\mathring{\mathbf{A}}_\Sigma$. Moreover,
$$
\Psi^* dV_{\Psi(\Sigma)}= W_{\mathbf{b}}^{-\frac{m}{2}} dV_\Sigma.
$$ 
By Theorem \ref{CRinv}, when $m\geq 2$ and $\Sigma$ is connected, the two functionals simultaneously vanish if and only if $\Sigma$ is a contact Whitney sphere. For Legendrian surfaces in $\mathbb{S}^5$,  $\mathcal{U}_{CR}$ is essentially the energy introduced by Wang in \cite{wangWillmoreLegendrianSurfaces2007}.
%
%
%It also follows that,
%\begin{align*}
%	(m+2)^2|{\mathbf{U}}_\Sigma^N|^2=(m+2)^2 | \mathbf{A}_\Sigma^N|^2-3(m+2) |\mathbf{H}^N_\Sigma|^2.
%\end{align*}
%That is,
%$$
%| \mathbf{A}_\Sigma^N|^2= |{\mathbf{U}}_\Sigma^N|^2+\frac{3}{m+2} |\mathbf{H}^N_\Sigma|^2.
%$$
%As
%$$
%|\mathring{\mathbf{A}}_\Sigma^{\hat{N}}|^2= |\mathbf{A}_\Sigma^{\hat{N}}|^2-\frac{1}{m}|\mathbf{H}_\Sigma^{\hat{N}}|^2 
%$$
%Hence,
%\begin{align*}
%	|\mathbf{A}_\Sigma|^2&=  | \mathbf{A}_\Sigma^N|^2+|\mathbf{A}_\Sigma^{\hat{N}}|^2= |{\mathbf{U}}_\Sigma^N|^2+\frac{3}{m+2} |\mathbf{H}^N_\Sigma|^2+|\mathring{\mathbf{A}}_\Sigma^{\hat{N}}|^2+\frac{1}{m}|\mathbf{H}_\Sigma^{\hat{N}}|^2 \\
%	&=|{\mathbf{U}}_\Sigma^N|^2+|\mathring{\mathbf{A}}_\Sigma^{\hat{N}}|^2+\frac{1}{m}|\mathbf{H}_\Sigma|^2+\frac{2(m-1)}{m(m+2)}|\mathbf{H}^N_\Sigma|^2
%\end{align*}

%$$
%\mathrm{R}_\Sigma=m(m-1)+|\mathbf{H}_\Sigma|^2-|\mathbf{A}_{\Sigma}|^2.
%$$
%In particular,

It follows from \eqref{GaussEqnUeqn}  and Lemma \ref{RicciULem}  that
\begin{align*}
	\mathrm{R}_\Sigma&=m(m-1)+\frac{m-1}{m} |\mathbf{H}_\Sigma|^2 - |{\mathbf{U}}_\Sigma^N|^2-|\mathring{\mathbf{A}}_\Sigma^{\hat{N}}|^2-\frac{2(m-1)}{m(m+2)}|\mathbf{H}^N_\Sigma|^2.
\end{align*}
When $m=2$ this can be rewritten as
$$
|{\mathbf{U}}_\Sigma^N|^2 +|\mathring{\mathbf{A}}_\Sigma^{\hat{N}}|^2=2-R_{\Sigma}+ \frac{1}{2} |\mathbf{H}_\Sigma|^2 - \frac{1}{4} |\mathbf{H}_\Sigma^{{N}}|^2=2-R_{\Sigma}+ \frac{1}{4} |\mathbf{H}_\Sigma^{{N}}|^2+ \frac{1}{2} |\mathbf{H}_\Sigma^{\hat{N}}|^2.
$$
For  a  closed, orientable surface of genus $g$, $\Sigma$, the Gauss-Bonnet Theorem implies
$$
\mathcal{U}_{CR}[\Sigma]+\mathcal{B}_{CR}[\Sigma]= \int_{\Sigma}1+ \frac{1}{8}|\mathbf{H}_\Sigma^{{N}}|^2 + \frac{1}{4} |\mathbf{H}_\Sigma^{\hat{N}}|^2 dA_\Sigma+ 4\pi (g-1).
$$
%When $\Sigma$ is Legendrian this simplifies to 
%$$
%\mathcal{U}_{CR}[\Sigma]= \int_{\Sigma}1+ \frac{1}{8} |\mathbf{H}_\Sigma|^2 dA_\Sigma+ 4\pi (g-1).
%$$
This leads us to introduce a functional generalizing \eqref{CRWillmoreEqn}
$$
\mathcal{W}_{CR}[\Sigma]=\int_{\Sigma}1+  \frac{1}{8}|\mathbf{H}_\Sigma^{{N}}|^2 + \frac{1}{4} |\mathbf{H}_\Sigma^{\hat{N}}|^2 dA_\Sigma=\int_{\Sigma}1+  \frac{1}{8}|\mathbf{H}_\Sigma|^2 + \frac{1}{8} |\mathbf{H}_\Sigma^{\hat{N}}|^2 dA_\Sigma, 
$$
so that when $\Sigma$ is closed and orientable
 $$\mathcal{W}_{CR}[\Sigma]=\mathcal{U}_{CR}[\Sigma]+\mathcal{B}_{CR}[\Sigma]+4\pi (1-g).$$ 
 Hence, this is also an $Aut_{CR}(\mathbb{S}^{2n+1})$ invariant energy.  
We now prove Theorem \ref{intro_thm_lag}.
\begin{proof}[Proof of Theorem \ref{intro_thm_lag}]
We have already established the invariance of $\mathcal{W}_{CR}$ for Legendrian surfaces.  This invariance and \eqref{CRWillmoreEqn} imply
$$
\mathcal{W}_{CR}[\Sigma]=\mathcal{W}_{CR}[\Psi(\Sigma)]\geq |\Psi(\Sigma)|_{\mathbb{S}}
$$
for all $\Psi\in \mathrm{Aut}_{CR}(\mathbb{S}^5)$.  
Hence, taking a sup gives $\mathcal{W}_{CR}[\Sigma]\geq \lambda_{CR}[\Sigma].$
When $\Sigma$ is embedded Theorem \ref{intro_thm_hori} implies there is $\Psi_0\in \mathrm{Aut}_{CR}(\mathbb{S}^5)$ so
$$
\lambda_{CR}[\Sigma]=|\Psi_0(\Sigma)|_{\mathbb{S}}
$$
and one has equality with $\mathcal{W}_{CR}[\Sigma]$ if and only if $\Psi_0(\Sigma)$ is minimal.
\end{proof}

Finally, we prove Proposition \ref{CRWillmoreHexTorProp}.
\begin{proof}[Proof of Proposition \ref{CRWillmoreHexTorProp}]
 Let 
 $\phi_H: \mathbb{C}/\Lambda_H\to \Sigma_H\subset \mathbb{S}^5$ be the conformal parameterization of the hexagonal torus and set $g_H=\phi^* g_{\mathbb{S}}$, which is a flat metric. 
By a direct computation,  for instance \cite[Equation (3.12)]{montielMinimalImmersionsSurfaces1986}, 
\begin{align*}
 |\Sigma_H|_{\mathbb{S}}&= \lambda_{CR}[\Sigma_H]=V_{CR}(5, \phi_H)=V_{CR}(5,  \mathbb{C}/\Lambda_H\, [g_H])\\
 &=\frac{1}{2}\lambda_1(g_H) | \mathbb{C}/\Lambda_H|_{g_H}=\frac{4\sqrt{3}}{3}\pi^2.
\end{align*}
Hence, by Proposition \ref{CRVolumeLowBoundProp} if $\Sigma$ is any other Legendrian torus parameterized by a conformal immersion $\phi:  \mathbb{C}/\Lambda_H\to \Sigma$ one has
$$
\lambda_{CR}[\Sigma]=V_{CR}(5, \phi)\geq V_{CR}(5, \mathbb{C}/\Lambda_H, [g_H])\geq  |\Sigma_H|_{\mathbb{S}}.
$$
 Theorem \ref{intro_thm_hori} implies either the inequality is strict or $\phi$ is an embedding as
$$
\frac{4\sqrt{3}}{3}\pi^2<8\pi =2|\mathbb{S}^2|_{\Real}.
$$
Hence, by Theorem \ref{intro_thm_lag}, 
$$
\mathcal{W}_{CR}[\Sigma]\geq \lambda_{CR}[\Sigma]\geq  |\Sigma_H|_{\mathbb{S}}
$$
 and equality holds  if and only if $\Psi_0(\Sigma)$ is minimal for some $\Psi_0\in Aut_{CR}(\mathbb{S}^5)$.

  In this last case, $\phi'=\Psi_0\circ \phi$ is an embedding by first eigenfunctions of $\mathbb{C}/\Lambda_H$ with its flat metric into $\mathbb{S}^5$.  As the multiplicity of the first eigenfunction of this space is exactly six -- see for instance \cite[Proposition 5]{montielMinimalImmersionsSurfaces1986} -- and $\phi_H$ is also such an embedding, and both are horizontal, there is a unitary matrix $A\in \mathbf{U}(6)$ so $\phi'=\Psi_A\circ \phi_H$.  That is, $\Sigma=\Psi(\Sigma_H)$ for some $\Psi\in Aut_{CR}(\mathbb{S}^5)$.
\end{proof}
%\begin{rem}
%	We note that the proof of Proposition \ref{CRWillmoreHexTorProp} shows one has equality throughout, then $\Sigma$ is a minimal Legendrian torus.  It should be possible to adapt the ideas of \cite{???} to show it is isometric to $\Sigma_H$ but we do not pursue this.
%\end{rem}
\subsection{Relation with an energy of Montiel-Urbano}\label{MontielUrbanoSec}
Let $\Pi: \mathbb{S}^{2n+1}\to \mathbb{CP}^n$ be the Hopf fibration satisfying  $\hat{\mathbf{T}}_p\in \ker D\Pi_p$.  This is a Riemannian submersion between $g_{\mathbb{S}}$ and $g_{\mathbb{CP}}$, the Fubini-Study metric.   As established in  \cite{reckziegelCorrespondenceHorizontalSubmanifolds1988}, $\phi: \Sigma\to  \mathbb{S}^{2n+1}$ is a Legendrian immersion if and only if $\Pi \circ \phi$ is a Lagrangian immersion into $\mathbb{CP}^n$.  Moreover,  the extrinsic geometries of $\phi$ and $\Pi\circ \phi$ correspond.  For instance, when $\phi: \mathbb{S}^n \to \mathbb{S}^{2n+1}$ is an embedding of the standard Legendrian $\mathbb{S}^n$, $\Pi\circ \phi: \mathbb{S}^n \to \mathbb{RP}^n \subset \mathbb{CP}^n$ is the double cover of a totally geodesic Lagrangian $\mathbb{RP}^n$.  Likewise, the contact Whitney spheres project to the Whitney spheres in $\mathbb{CP}^n$. 

 In general,  \cite{reckziegelCorrespondenceHorizontalSubmanifolds1988} implies any Lagrangian immersion into $\mathbb{CP}^n$ can be locally lifted along $\Pi$ to a Legendrian embedding into $\mathbb{S}^{2n+1}$, but there are global obstructions.  A general condition ensuring a Lagrangian immersion $\phi: M^n\to \mathbb{CP}^n$ may be lifted to a Legendrian embedding $\tilde{\phi}: M^n\to \mathbb{S}^{2n+1}$ is given in \cite{baldridgeLiftingLagrangianImmersions2022}.  In \cite{RizellGolovko}, the authors study when a Lagrangian torus in $\mathbb{CP}^n$ admits a cover that can be lifted to a Legendrian embedding in $\mathbb{S}^{2n+1}$ --  these covers are examples of what they call \emph{Bohr-Sommerfeld} Lagrangian immersions.

In  \cite{montielWillmoreFunctionalCompact2002},  Montiel-Urbano introduce an energy, $W^-$,  of surfaces in $\mathbb{CP}^2$ related to the Willmore energy.  
A consequence of the correspondence of \cite{reckziegelCorrespondenceHorizontalSubmanifolds1988} is that, for a Legendrian surface $\Sigma\subset \mathbb{S}^5$, counting multiplicities,
$$\mathcal{W}_{CR}[\Sigma]=\frac{1}{2}W^-[\Pi(\Sigma)].$$
As a consequence, we obtain from Theorem \ref{intro_thm_hori} a special case of \cite[Corollary 6] {montielWillmoreFunctionalCompact2002}.
\begin{prop}
	If $\phi: \mathbb{S}^2\to \mathbb{CP}^2$ is a Lagrangian immersion, then
	$$
	W^-[\phi]\geq 8\pi
	$$
	with equality if and only if $\phi$ parameterizes a Whitney sphere or is the double cover of a totally geodesic $\mathbb{RP}^2$.
\end{prop}
\begin{proof}
	As $\mathbb{S}^2$ is simply connected, the lifting of \cite{reckziegelCorrespondenceHorizontalSubmanifolds1988} yields a Legendrian immersion
	$$
	\tilde{\phi}: \mathbb{S}^2\to \mathbb{S}^5 \mbox{ with }	\phi=\Pi\circ\tilde{\phi}.
	$$
  By Theorem \ref{intro_thm_hori},
	$$
	W^-[\phi]=2\mathcal{W}_{CR}[\tilde{\phi}]\geq 2 \lambda_{CR}[\tilde{\phi}]\geq 8\pi
	$$
 with equality only when $\tilde{\phi}$ parameterizes a contact Whitney sphere.  Hence,  $\phi$ parameterizes a Whitney sphere or double covers a totally geodesic $\mathbb{RP}^2$.
\end{proof}

Finally, as observed in the introduction, the parameterization of the hexagonal torus, $\phi_H$,  induces a three-fold cover of the Lagrangian Clifford torus $\Sigma_{C}^{\mathbb{CP}}$.
Thus,
$$
\mathcal{W}_{CR}[\Sigma_H]=\mathcal{W}_{CR}[\phi_H]=\frac{1}{2}W^-[\Pi\circ \phi_H]=\frac{3}{2} W^-[\Sigma_{C}^{\mathbb{CP}}].
$$
It follows from \cite[Theorem 1.3]{RizellGolovko} that any monotone Lagrangian torus in $\mathbb{CP}^2$ admits a three-fold cover that lifts along $\Pi$ to a Legendrian torus in $\mathbb{S}^5$.  Being monotone is a natural condition in this setting and holds for any Lagrangian torus Hamiltonian isotopic to $\Sigma_{C}^{\mathbb{CP}}$ -- see \cite[pg. 55]{Joyce} and \cite{EvansLCP} for further discussion of this condition in related geometric settings.
Hence, if conjecture \ref{CRWillmoreConj} holds, then any monotone Lagrangian torus $\Sigma \subset \mathbb{CP}^2$ and in particular any surface Hamiltonian isotopic to $\Sigma_{C}^{\mathbb{CP}}$ satisfies
	$$
	W^-[\Sigma]\geq W^-[\Sigma_{C}^{\mathbb{CP}}].
$$
This is a special case of a  conjecture of Montiel-Urbano \cite[Remark 3]{montielWillmoreFunctionalCompact2002} who posit this inequality for all Lagrangian tori (and even ask if it holds for all tori).

\appendix
\section{Geometric Computations}

\subsection{Rank one deformation}
Let $(M,g)$ be a Riemannian manifold and suppose that $\tau$ is a smooth one form on $M$ and $\alpha$ a smooth, possibly non-positive,  function.  We suppose that $|\tau|_g^2 \alpha>-1$.  We may then define a new Riemannian metric
$$
h=g+\alpha \tau\otimes \tau.
$$
Denote by $\mathbf{T}$  the vector field satisfying
$$
g(\mathbf{T}, Z)=\tau(Z).
$$
Let $\nabla^g$ denote the Levi-Civita connection of $g$ and let $\nabla^h$ be the Levi-Civita connection of $h$ and denote their difference by the $(1,2)$ tensor field $C$ given by
$$
\nabla_{X}^h Y -\nabla_X^g Y =C(X,Y).
$$

In what follows it will be convenient to assume $\mathbf{T}$ satisfies the following hypotheses:  
\begin{equation}\label{THypeqn}
	|\mathbf{T}|_g=1,  \nabla_{\mathbf{T}}^g \mathbf{T}=0 \mbox{ and } 
	\nabla_{X}^g \mathbf{T}= \mathbf{a}(X) \mbox{ for $X$ $g$-orthogonal to $\mathbf{T}$}.
\end{equation}
Here $\mathbf{a}$ is a $(1,1)$ tensor field satisfying
$$
g(\mathbf{a}(X), Y)=-g(X,\mathbf{a}(Y)) \mbox{ and } g(\mathbf{a}(X), \mathbf{T})=0
$$
for $X,Y$ orthogonal to $\mathbf{T}$. These are conditions one expects for the Reeb vector field of a Sasakian almost contact metric structure.

We recall the following computation from \cite[Appendix B]{BernBhattCplxHyp}.
\begin{prop}\label{RankOneProp}
Suppose $\mathbf{T}$ satisfies \eqref{THypeqn}.  When $X,Y$ are orthogonal to $\mathbf{T}$, the tensor field $C$ satisfies
$$
C(X,Y)=\mathbf{0}. 
$$
\end{prop}
\begin{proof}
	In general, it follows from  \cite[Appendix B]{BernBhattCplxHyp} that the tensor field $C$ satisfies
	$$
	C(X,Y)=\frac{|\mathbf{T}|^{-2}_g}{1+\alpha|\mathbf{T}|_g^2} c(X,Y, \mathbf{T}) \mathbf{T}+\sum_{i=1}^n c(X,Y, E_i) E_i
	$$
	where  $E_1, \ldots, E_{n}$ are orthonormal vectors that are also orthogonal to $\mathbf{T}$ and
	$c$ is a $(0,3)$ tensor field satisfying
	\begin{align*}
		c(X,&Y, \mathbf{T})=\frac{1}{2} \big(g(\nabla^g_X \alpha \mathbf{T}, Y)|\mathbf{T}|^{2}_g+g(\nabla^g_Y \alpha\mathbf{T}, X)|\mathbf{T}|^{2}_g+\alpha  g(\nabla_X^g \mathbf{T}, \mathbf{T})g(\mathbf{T}, Y)\\ 
		&+\alpha  g(\nabla_Y^g \mathbf{T}, \mathbf{T})g(\mathbf{T}, X)-g(\nabla_{\mathbf{T}}^g \alpha \mathbf{T}, X)g(\mathbf{T}, Y)-\alpha g(\nabla_{\mathbf{T}}^g  \mathbf{T}, Y)g(\mathbf{T},X) \big);\\
		c(X,&Y, E_i)= \frac{1}{2} \alpha g(\nabla_X^g \mathbf{T}, E_i) g(\mathbf{T},Y)+\frac{1}{2} \alpha g(\nabla_Y^g \mathbf{T}, E_i) g(\mathbf{T},X)\\
		&-\frac{1}{2} g(\nabla_{E_i}^g \alpha \mathbf{T}, X) g(\mathbf{T},Y)-\frac{1}{2} g(\nabla_{E_i}^g \alpha \mathbf{T}, Y) g(\mathbf{T},X).
	\end{align*}
	The simplification when \eqref{THypeqn} holds is immediate.
\end{proof}

\subsection{Rank two deformations}\label{RankTwoSec}
Let $(M,g)$ be a Riemannian manifold and suppose that $\tau$ and $\omega$ are smooth one forms on $M$ that satisfy $\langle \tau, \omega\rangle_g=0$ and $|\tau|_g |\omega_g|<1$.  We may then define a new Riemannian metric
$$
h=g+ \tau\otimes \omega+\omega\otimes \tau.
$$

Let $\nabla^g$ denote the Levi-Civita connection of $g$ and $\nabla^h$ denote the Levi-Civita connection of $h$.
Let $C(X,Y)$ be the $(1,2)$ tensor field given by
$$
C(X,Y)=\nabla^h_{X} Y-\nabla^g_X Y.
$$
Let $\mathbf{T}$ and $\mathbf{S}$ be the vector fields satisfying
$$
g(\mathbf{T}, Z)=\tau(Z) \mbox{ and }
g(\mathbf{S}, Z)=\omega(Z).
$$
The hypotheses ensure that $g(\mathbf{T}, \mathbf{S})=0$ and $|\mathbf{S}|_g|\mathbf{T}|_g<1$.  
\begin{prop}\label{RankTwoProp}
Suppose that $\mathbf{T}$ satisfies \eqref{THypeqn}. For $X,Y$ $g$-orthogonal to $\mathbf{T}$,
\begin{align*}
	C(X,Y)&=\frac{\mathbf{T}-\mathbf{S}}{2(1-|\mathbf{S}|_g^2)}\big(   g(\nabla_X^g \mathbf{S}, Y)+g(\nabla_Y^g \mathbf{S}, X) +2g(\mathbf{a}(\mathbf{S}), X)g (\mathbf{S}, Y)\\
	&+2g(\mathbf{a}(\mathbf{S}), Y)g (\mathbf{S},X)
	\big)+ g(\mathbf{S}, Y)\mathbf{a}(X)+g(\mathbf{S}, X)\mathbf{a}(Y).
\end{align*}	
\end{prop}
\begin{rem}
When $|\mathbf{T}|_g=1$, $X$ is $g$-orthogonal to $\mathbf{T}$ if and only if it is $h$-orthogonal to $\mathbf{T}-\mathbf{S}$.  Indeed, for any $X$,
\begin{align*}
	h(\mathbf{T}-\mathbf{S},X)&=g(\mathbf{T}-\mathbf{S},X)+|\mathbf{T}|_g^2 g(\mathbf{S},X)-|\mathbf{S}|^2_g g(\mathbf{T},X)\\
	&=(1-|\mathbf{S}|^2_g) g(\mathbf{T},X).
\end{align*}
By hypothesis $|\mathbf{S}|_g=|\mathbf{S}|_g|\mathbf{T}|_g<1$ and the claim holds.
\end{rem}
\begin{proof}
Using the Koszul formula we compute that
\begin{align*}
	g(C(X,Y),Z)&+ \tau(C(X,Y))\omega(Z)+\omega(C(X,Y)) \tau(Z)= c(X,Y,Z)
\end{align*}
where
\begin{align*}
	c(X,Y,Z)&=\frac{1}{2} (\nabla^g_{X} \tau)(Y)\omega(Z)+\frac{1}{2} (\nabla^g_{Y} \tau)(X) \omega(Z)+\frac{1}{2} (\nabla^g_X \tau)(Z) \omega(Y)\\
	&+\frac{1}{2} (\nabla^g_Y \tau)(Z) \omega(X)+\frac{1}{2} (\nabla^g_X \omega)(Z) \tau(Y)+\frac{1}{2} (\nabla^g_Y \omega)(Z) \tau(X)\\
	&+ \frac{1}{2} (\nabla^g_{X} \omega)(Y)\tau(Z)+\frac{1}{2} (\nabla^g_{Y} \omega)(X) \tau(Z)-\frac{1}{2} (\nabla^g_Z \omega)(X) \tau(Y)\\
	&-\frac{1}{2} (\nabla^g_Z \omega)(Y) \tau(X)-\frac{1}{2} (\nabla^g_Z \tau)(X) \omega(Y)-\frac{1}{2} (\nabla^g_Z \tau)(Y) \omega(X).	
\end{align*}

Choose orthonormal vectors $E_1, \ldots, E_n$ that are also orthogonal to $\mathbf{S}$ and $\mathbf{T}$. 
$$
C(X,Y)=a(X,Y) \mathbf{T}+b(X,Y) \mathbf{S}+\sum_{i=1}^n c_i(X,Y) E_i
$$
where $a,b$, and $c_i$ are symmetric $(0,2)$ forms given by
$$
a(X,Y)= \frac{|\mathbf{T}|_g^{-2}}{(1-|\mathbf{S}|_g^2|\mathbf{T}|_g^2)}  c(X,Y, \mathbf{T})-\frac{1}{(1-|\mathbf{S}|_g^2 |\mathbf{T}|_g^2)} c(X,Y, \mathbf{S}),
$$
$$
b(X,Y)= \frac{ c(X,Y, \mathbf{S})-c(X,Y, \mathbf{T})}{(1-|\mathbf{S}|_g^2 |\mathbf{T}|_g^2)}+  \frac{|\mathbf{S}|^{-2}-|\mathbf{T}|^{-2}}{(1-|\mathbf{S}|_g^2|\mathbf{T}|_g^2)}  c(X,Y, \mathbf{S}),
$$ 
$$
c_i(X,Y) = c(X,Y,E_i).
$$
Using condition \eqref{THypeqn} one computes for $X,Y$ $g$-orthogonal to $\mathbf{T}$:
\begin{align*}
	c(X,Y, &\mathbf{T})= \frac{1}{2}\big( g(\nabla_X^g \mathbf{S}, Y)|\mathbf{T}|^2+g(\nabla_Y^g \mathbf{S}, X)|\mathbf{T}|^2+g(\nabla_X^g \mathbf{T}, \mathbf{T})g( \mathbf{S},Y)\\
	&+g(\nabla_Y^g \mathbf{T}, \mathbf{T})g( \mathbf{S},X)+g(\nabla_X^g \mathbf{S}, \mathbf{T})g(\mathbf{T},Y)+g(\nabla_Y^g \mathbf{S}, \mathbf{T})g(\mathbf{T},X)\\
	&-g(\nabla_{\mathbf{T}}^g \mathbf{S}, X)g (\mathbf{T}, Y)-g(\nabla_{\mathbf{T}}^g \mathbf{S}, Y)g (\mathbf{T}, X)-g(\nabla_{\mathbf{T}}^g \mathbf{T}, X)g (\mathbf{S}, Y)\\
	&-g(\nabla_{\mathbf{T}}^g \mathbf{T}, Y)g (\mathbf{S},X)\big)\\
	&=\frac{1}{2}\big( g(\nabla_X^g \mathbf{S}, Y)+g(\nabla_Y^g \mathbf{S}, X)\big),
\end{align*}
\begin{align*}
	c(X,Y,&\mathbf{S})= \frac{1}{2}\big( g(\nabla_X^g \mathbf{T}, Y)|\mathbf{S}|^2+g(\nabla_Y^g \mathbf{T}, X)|\mathbf{S}|^2+g(\nabla^g_X \mathbf{T}, \mathbf{S})g(\mathbf{S},Y)\\
	&+g(\nabla^g_Y \mathbf{T}, \mathbf{S})g(\mathbf{S},X) +g(\nabla^g_X \mathbf{S}, \mathbf{S})g(\mathbf{T},Y)	+g(\nabla^g_Y \mathbf{S}, \mathbf{S})g(\mathbf{T},X)\\
	&-g(\nabla_{\mathbf{S}}^g \mathbf{S}, X)g (\mathbf{T}, Y)-g(\nabla_{\mathbf{S}}^g \mathbf{S}, Y)g (\mathbf{T}, X)-g(\nabla_{\mathbf{S}}^g \mathbf{T}, X)g (\mathbf{S}, Y)\\
	&-g(\nabla_{\mathbf{S}}^g \mathbf{T}, Y)g (\mathbf{S},X)\big)\\
	&=-g(\mathbf{a}( \mathbf{S}), X)g (\mathbf{S}, Y)-g(\mathbf{a}(\mathbf{S}), Y)g (\mathbf{S},X)\big),
\end{align*}
\begin{align*}
	c(X,Y,&E_i)=\frac{1}{2}\big( g(\nabla_X^g \mathbf{T}, E_i) g(\mathbf{S},Y)+g(\nabla_Y^g \mathbf{T}, E_i) g(\mathbf{S},X)\\
	&+g(\nabla_X^g \mathbf{S}, E_i) g(\mathbf{T},Y)+g(\nabla_X^g \mathbf{S}, E_i) g(\mathbf{T},Y)- g(\nabla_{E_i}^g \mathbf{T}, X) g(\mathbf{S}, Y)\\
	&- g(\nabla_{E_i}^g \mathbf{T}, Y) g(\mathbf{S}, X)- g(\nabla_{E_i}^g \mathbf{S}, X) g(\mathbf{T}, Y)-g(\nabla_{E_i}^g \mathbf{S}, Y) g(\mathbf{T}, X)\big)\\
	&= \mathbf{a}(X), E_i) g(\mathbf{S},Y)+g( \mathbf{a}(Y), E_i) g(\mathbf{S},X)
\end{align*}
where we repeatedly used the antisymmetry property of $\mathbf{a}$. 
The result follows readily from this.
\end{proof}

\section{Integral Expansions}
We record some asymptotic expansions for certain integrals we use.

\begin{lem}\label{horizontalExpansionLem}
	For $\tau\in(1,\infty), a\in (0,\infty)$ and integers $l,k\geq 1$, let
	$$
 J_{k, l}(\tau; a)=  	\int_0^a \frac{r^{l-1}}{(1+\tau  r^2)^{k}} dr.
	$$
	The following asymptotic expansions hold for $a$ fixed and $\tau\to \infty$,
	$$
		J_{k,l}(\tau; a)  = \left\{\begin{array}{cc} c_{k, l} \tau^{-\frac{1}{2}l} +O(\tau^{-\frac{1}{2}(l+3)}; a) &  l<2 k-2\\
			c_{k, l} \tau^{-\frac{1}{2}l} +O(\tau^{-\frac{1}{2}(l+2)}; a) &  l=2 k-2\\
		c_{k,l} \tau^{-\frac{1}{2}l} +O(\tau^{-\frac{1}{2}(l+1)}; a) &  l=2k-1\\
		\frac{1}{2} \tau^{-\frac{1}{2}l} \log \tau +  O(\tau^{-\frac{1}{2}l} ; a) & l=2k\\
		O(\tau^{-k}; a) & l>2k.\end{array}\right.
	$$
\end{lem}
\begin{proof}
	We readily compute by a change of variables formula that
	$$
	J_{1,2}(\tau; a)= \frac{1}{2}\tau^{-1}\log (a^2 \tau+1)=\frac{1}{2} \tau^{-1}\log \tau +O(\tau^{-1}; a)
	$$
	while, for $k\geq 2$, 
	$$
	J_{k, 2}(\tau; a)= 	 \frac{1}{2(k-1)} \tau^{-1}\left( 1-(1+\tau a^2)^{1-k}\right)=	\frac{1}{2(k-1)}\tau^{-1}+O(\tau^{-k}; a, k) .
$$
	One also computes that
	$$
		J_{1,1}(\tau; a)= \frac{\tan^{-1}(a\sqrt{\tau})}{\tau^{1/2}}=\frac{\pi}{2}\tau^{-\frac{1}{2}}+O(\tau^{-1}; a).
	$$
	Integrating by parts, yields, for $k\geq 1$,  the following recursion formula
	$$
	J_{k+1, 1}(\tau; a)=\frac{2k-1}{2k} 	J_{k, 1}(\tau; a)+\frac{1}{2k} \frac{a}{(1+\tau a^2)^k}.
	$$
	We also observe that when $k=2$ there is a favorable cancellation and so one obtains
	$$
	J_{2, 1}(\tau; a)= \frac{\pi}{4} \tau^{-\frac{1}{2}} + O(\tau^{-2};a).
	$$
%	It follows by induction that, for $k\geq 2$.
%	$$
%	J_{1,k}(\tau; a)= 4^{1-k} {2k-3 \choose k-1}\pi\tau^{-\frac{1}{2}}+ O(\tau^{-2}; a).
%	$$

	When $l\geq 3$, manipulating the integral gives a recursive identity
	$$
	J_{k, l}(\tau; a)=\tau^{-1} J_{k-1, l-2}(\tau; a)-\tau^{-1}J_{k, l-2}(\tau; a).
	$$
	A direct integration yields, for $l\geq 1$, 
	$
	J_{0, l}(\tau; a)=\frac{1}{l} a^l.
	$
	Using these two facts and iterating, 
	we deduce that
	$$
		J_{1, l}(\tau; a)=\frac{1}{\tau (l-2)} (a^{l-2}-(l-2)	J_{1,l-2}(\tau; a))=\frac{a^{l-2}}{l-2} \tau^{-1} +O(\tau^{-\frac{3}{2}}; a, l).
	$$
Finally, use the recursive identity to reduce to the computed cases.
\end{proof}

\begin{lem}\label{horizontalWithtLem}
	Fix integers $k, l\geq 1$. For $t\in (0,1)$ and $\epsilon\in (0,1)$ set
	$$
	I_{k,l}(t; \epsilon)=\int_{\epsilon}^1 \frac{(1-x)^{\frac{l}{2}-1}}{(1-(1-t) x)^k} dx.
	$$
	One has the following asymptotics as $t\to 0^+$:
	$$
	I_{k,l}(t; \epsilon)= \left\{\begin{array}{cc} C_{k,l}(1+\frac{l}{2} t) t^{-k+\frac{1}{2}l} +O(t^{-k+\frac{1}{2}(l+3)}; \epsilon) &  l<2 k-2\\
				C_{k,l}t^{-1} +O(1; \epsilon) & l=2k-2\\
		C_{k,l}t^{-\frac{1}{2}} +O(1; \epsilon) &  l=2k-1\\
%		C_{k,l} (1+\frac{l}{2}  t)t^{-k+\frac{1}{2}l} +O(t^{-k+\frac{1}{2}(l+3)}; \epsilon)& l=2m<2k\\
		-  \log t +  O(1; \epsilon) & l=2k\\
		O(1; \epsilon) & l>2k.\end{array}\right.
	$$
	Moreover, the coefficients  satisfy the following:
	\begin{equation*}
   C_{k,l}=\left\{\begin{array}{cc} % \frac{\pi}{2}  & k=2, l=3  \\ 
             C_{k-1, l-2}-C_{k,l-2} & 2k>l\geq 3\\
         \frac{l-2}{2(k-1)} C_{k-1, l-2} &  2k\geq l\geq 3       .
       \end{array}\right.
	\end{equation*}
\end{lem}
\begin{proof}
Setting $r=\sqrt{1-x}$, the change of variables formula imply
	$$
	I_{k,l}(t; \epsilon)= 2 t^{-k} J_{k,l}(t^{-1}-1; \sqrt{1-\epsilon}).
	$$
The leading order behavior of expansion then follows from Lemma \ref{horizontalExpansionLem}.  The subleading order behavior when $l<2k-2$ follows from
$$
(t^{-1}-1)^{-\frac{l}{2}}= t^{-\frac{l}{2}} +\frac{l}{2}t^{-\frac{l}{2}+1}+O(t^{-\frac{l}{2}+2}).
$$

By algebraically manipulating the integral one obtains the recursive relationship
	$$
I_{k,l}(t; \epsilon)= t (I_{k,l}(t; \epsilon)-I_{k,l-2}(t; \epsilon))+I_{k-1,l-2}(t; \epsilon)	.
	$$
When $2k>l\geq 3$ this formula together with the leading order behavior of the asymptotic expansion  yields
$$
C_{k,l}=C_{k-1,l-2}-C_{k,l-2}.
$$
Likewise, integrating by parts implies
$$
I_{k,l}(t; \epsilon)= \frac{2k}{l} I_{k+1, l+2}(t; \epsilon)-t\frac{2k}{l} I_{k+1,l+2}(t; \epsilon)+O(1; \epsilon).
$$
When $2k>l+2\geq 1$, the leading order behavior of the asymptotic expansion gives
$$
C_{k,l} = \frac{2k}{l} C_{k+1, l+2}.
$$ 
	
\end{proof}

%\bibliographystyle{hamsabbrv}
%\appendix
\bibliographystyle{hamsabbrv}
\bibliography{Library2}
\end{document}